\documentclass{amsart}
\usepackage{graphicx}
\usepackage{stmaryrd}

\newtheorem{acknowledgement}{Acknowledgement}
\newtheorem{thm}{Theorem}[section]

\newtheorem{lem}{Lemma}[section]
\newtheorem{prop}{Proposition}[section]
\newtheorem{defn}{Definition}[section]
\newtheorem{rem}{Remark}[section]

\newcommand{\norm}[1]{\left\Vert#1\right\Vert}

\newcommand{\abs}[1]{\left\vert#1\right\vert}

\newcommand{\Real}{\mathbb R}

\newcommand{\pfrac}[2]{\frac{\partial #1}{\partial #2}}
\newcommand{\wc}{\rightharpoonup}
\begin{document}
\title{Partial regularity of a minimizer of the relaxed energy for biharmonic maps}

\numberwithin{equation}{section}
\author{Min-Chun Hong and Hao Yin}

\address{Min-Chun Hong, Department of Mathematics, The University of Queensland\\
Brisbane, QLD 4072, Australia}  \email{hong@maths.uq.edu.au}

\address{Hao Yin, Department of Mathematics, The University of Queensland\\
Brisbane, QLD 4072, Australia and Department of Mathematics, Shanghai Jiaotong University, Shanghai, China}
\email{haoyin@sjtu.edu.cn}

\begin{abstract}
    In this paper, we study the relaxed energy for biharmonic maps from a $m$-dimensional domain into spheres. By an approximation method, we prove the existence of
    a minimizer of the relaxed energy of the Hessian energy,
    and that  the minimizer is biharmonic and smooth outside a singular set $\Sigma$ of finite $(m-4)$-dimensional Hausdorff measure.
     Moreover, when $m=5$, we   prove that the singular set $\Sigma$ is $1$-rectifiable.
\end{abstract}
\subjclass{AMS 35J48, 49J45} \keywords{relaxed energy, biharmonic
maps, partial regularity }

 \maketitle

\pagestyle{myheadings} \markright {} \markleft { }

\section{Introduction}

Let $\Omega$ be a bounded smooth domain in $\Real^m$ and $N$ a
compact manifold without boundary, which is  embedded in
$\Real^k$. For a map $u\in W^{2,2}(\Omega, N)$, we define its
Hessian energy by
\begin{equation}
    {\mathbb H}(u)=\int_{\Omega}\abs{\triangle u}^2 dx.
    \label{eqn:Henergy}
\end{equation}
 A  critical point of the
 Hessian energy functional in  $W^{2,2}(\Omega,N)$ is called a biharmonic map.

The partial regularity for stationary biharmonic maps has
attracted much attention. Motivated by the partial regularity
result for stationary harmonic maps (\cite {Be}), Chang, Wang and
Yang  in \cite{CWY}  introduced a study  of stationary biharmonic
maps and proved partial regularity of stationary biharmonic maps
into spheres. Wang in \cite{W3}  generalized  their result for
stationary biharmonic maps into a compact manifold $N$. Recently,
the regularity problem for stationary biharmonic maps was
revisited by Struwe  in \cite{Struwe}  from a new point of view.
Typical stationary biharmonic maps are minimizing biharmonic maps.
 The first author  and Wang
 in \cite{HW} proved that the Hausdorff dimension of the singular set of minimizing biharmonic maps into spheres is at most
 $m-5$. Recently, Scheven in \cite {Scheven} generalized the result for minimizing biharmonic maps into a
 general manifold $N$. This is an analogous result to the optimal partial regularity for
minimizing harmonic maps due to Giaquinta-Giusti  \cite{GG}  and
Schoen-Uhlenbeck  \cite{SU}.

On the other hand,  motivated by a gap phenomenon for the
Dirichlet energy discovered by Hardt-Lin (\cite{HL}),   Bethuel,
Brezis and Coron in \cite{BBC}  introduced  a relaxed energy for
the Dirichlet energy of maps in $W^{1,2}(B^3, S^2)$ and proved
that a minimizer of the relaxed energy is a  harmonic map.
Giaquinta, Modica and Soucek in \cite {GMS1} proved the partial
regularity of the minimizers of the relaxed energy for harmonic
maps.    A similar gap phenomenon for Hessian energy functional to
the one for the Dirichlet energy was observed in   \cite{HW}.
More precisely, there is a smooth domain $\Omega$ in $\Real^5$ and
a boundary value map $\psi:\partial\Omega\to S^4$ such that
\begin{equation*}
    \min_{u\in W^{2,2}_{\psi}(\Omega,S^4)}{\mathbb H}(u)< \inf_{v\in W^{2,2}_\psi(\Omega,S^4)\cap C^0(\bar{\Omega},S^4)} {\mathbb H}(v).
\end{equation*}
 Following  the  context of harmonic maps  (see \cite {BB}),
 a family of $\lambda$-relaxed
  energy functionals for bi-harmonic maps  was considered  in \cite{HW}  in the
  following:
\begin{equation*}
    {\mathbb H}_\lambda(u)={\mathbb H}(u)+16 \lambda \sigma_4 L(u), \quad \forall u\in W^{2,2}_\phi(\Omega,S^4) \mbox{ and } \lambda\in [0,1],
\end{equation*}
where $\sigma_4$ is the area of the unit sphere $S^4\subset
\Real^5$ and
\begin{equation*}
    L(u)=\frac{1}{\sigma_4}\sup_{\xi:\Omega\to \Real, \norm{\nabla \xi}_{L^\infty}\leq 1}
    \left\{
    \int_{\Omega} D(u)\cdot \nabla \xi dx-\int_{\partial \Omega}D(u)\cdot\nu\xi dH^{n-1}
    \right\}
\end{equation*}
for the $D$-field    $D(u)$. Moreover, it was proved in \cite{HW}
that ${\mathbb H}_\lambda$ are sequentially lower semi-continuous
and that their minimizers are partially regular biharmonic maps
for $\lambda\in [0,1)$. However, it is not known whether ${\mathbb
H}_1(u)$ is a relaxed energy for the Hessian functional or not.
Thus, there is an open question on the existence and partial
regularity of minimizers of the relaxed energy for biharmonic
maps.

In order to define a relaxed energy for biharmonic maps,  we
denote by $W_{u_0}^{2,2}(\Omega,S^n)$ the set of all maps $u\in
W^{2,2}(\Omega,S^n)$  satisfying the boundary condition
\begin{equation}\label{eqn:boundary}
    u-u_0|_{\partial \Omega}=0,\quad \nabla(u-u_0)|_{\partial\Omega}=0,
\end{equation}
where $u_0$ is smooth on $\overline{\Omega}$.  Similarly, we
denote by $C^\infty_{u_0}(\Omega, S^n)$ the space of smooth maps
satisfying (\ref{eqn:boundary}). Following a strategy in \cite
{GMS2}, we can define the relaxed energy $F(u)$ of biharmonic maps
in an abstract way; i.e.
\begin{defn}\label{defn:F}
    For each $u\in W^{2,2}_{u_0}(\Omega,S^n)$, we define the relaxed energy $F(u)$ by
    \begin{equation*}
        F(u)=\inf\left\{ \liminf_{k\to \infty} {\mathbb H}(u_k) | \quad \{u_k\}\subset C^\infty_{u_0}(\Omega,S^n),
        u_k\wc u \mbox{ weakly in } W^{2,2}(\Omega,S^n) \right\}.
    \end{equation*}
\end{defn}
 It can be proved  (see below Lemmas 2.1-2.2) that there is a minimizer of $F$ in
$W^{2,2}_{u_0}(\Omega, S^n)$  and
\begin{equation}\label{eqn:relax}
    \min_{u\in W^{2,2}_{u_0}(\Omega,S^n)}F(u) =\inf_{u\in W^{2,2}_{u_0}(\Omega,S^n)\cap C^0(\bar{\Omega},S^n)} {\mathbb
    H}(u).
\end{equation}
However, without the  explicit form of $F(u)$, we do not know how
to prove the partial regularity of a minimizer of $F$. To overcome
this difficulty, we consider a family of perturbed functionals
$\mathbb H_\varepsilon(\varepsilon>0)$ defined by
\begin{defn} For each $\varepsilon >0$, we define
    the perturbed functional ${\mathbb H}_\varepsilon: W^{2,2}_{u_0}\cap W^{1,m+1}(\Omega,S^n)\to \Real$ by
    \begin{equation*}
        {\mathbb H}_\varepsilon(u)=\int_{\Omega} \abs{\triangle u}^2+\varepsilon \abs{\nabla u}^{m+1}
        dx.
    \end{equation*}
\end{defn}
The similar approximation for the relaxed energy for harmonic maps
was recently studied by Giaquinta and the two authors in \cite {GHY}.

 The
first result of this paper is:
\begin{thm}\label{thm:first}
For each   $\varepsilon >0$, there exists  a minimizer
$u_\varepsilon$  of $\mathbb H_\varepsilon$ in the space
$W^{2,2}_{u_0}\cap W^{1,m+1}(\Omega, S^n)$. Then,  for each
sequence $\varepsilon\to 0$, there is
   a subsequence $\varepsilon_i$ such that
   $u_{\varepsilon_i}$ converges to a map $u$ weakly in $W^{2,2}(\Omega,
     S^n)$ and
     $u$ is a minimizer of the relaxed energy $F$ in $W_{u_0}^{2,2}((\Omega, S^n)$
      and  a biharmonic map.
     Moreover, the minimizer $u$ is smooth outside a relatively closed singular  set $\Sigma$, whose  $(m-4)$-Hausdorff
     measure is  finite,
     defined by
    \begin{equation*}
        \Sigma=\bigcap_{R>0} \left\{x\in \Omega\, |\quad B_R(x)\subset \Omega,\,
        \liminf_{\varepsilon_i\to 0} R^{4-m}\int_{B_R(x)} \abs{\triangle u_ {\varepsilon_i}}^2 dx
        \geq \varepsilon_0\right\}
    \end{equation*}
    for some constant $\varepsilon_0>0$.
\end{thm}
 It is well known that one of main difficulties in the proof of partial regularity of stationary biharmonic maps is
  that the monotonicity formula for biharmonic maps involves boundary terms of undetermined sign. Chang, Wang and Yang \cite{CWY}
  used a complicated iteration to deal with this difficulty. Struwe in \cite{Struwe} had a nice observation  and gave a simple proof, on which
  our proof to Theorem \ref{thm:first} is based.
  In fact, our proof to Theorem \ref{thm:first} is more complicated, since the limit $u$ of $u_\varepsilon$ is not stationary, so there is no
  `nice' monotonicity formula for $u$.
  Our approach is to prove a monotonicity formula for $u_\varepsilon$ and pass a limit of $\varepsilon\to 0$.

In Section \ref{sec:further}, we study further properties of the
boundary terms in the monotonicity formula. In particular, we show
that for $\mathcal H^{m-4}$ a.e. $x\in \Omega$, the quantity
\begin{equation*}
   \Theta (x)= \lim_{r\to 0} r^{4-m} \mu(B_r(x))
\end{equation*}
exists, where $\mu(B_r(x))=\lim_{\varepsilon\to
0}\int_{B_r(x)}\abs{\triangle u_\varepsilon}^2 dx$. This is an
interesting feature of the monotonicity formula for biharmonic
maps. Namely, although the boundary term of unknown sign spoils
the monotonicity of the scaled energy, the limit of the scaled
energy exists. Our proof also works for a sequence of stationary
biharmonic maps into any compact manifold. Thanks to a result of
Preiss \cite{Big}, we have
\begin{thm}\label{thm:second}
    Let $ \tilde u_i$ be a sequence of stationary biharmonic maps from $\Omega\subset \Real^m$ into a compact manifold $N\subset \Real^k$.
    Assume that $ \tilde u_i$ converges weakly to a map
    $ \tilde u$ in
    $W^{2,2}$ and
    \begin{equation*}
       \tilde \mu=\lim_{i\to \infty} \abs{\triangle  \tilde u_i}^2 dx = \abs{\triangle  \tilde u}^2 + \tilde \nu.
    \end{equation*}
  It can be shown  that  $ \tilde u_i$ converges smoothly
    to
    $ \tilde u$ in
    $\Omega \backslash   \tilde{\Sigma}$, where
\begin{equation*}
    \tilde{\Sigma}=\left \{x\in \Omega\,|\quad \liminf_{\rho\to 0} \left( \rho^{2-m}\int_{B_\rho(x)}\abs{\nabla \tilde{u}}^2 dx
    +\rho^{4-m}\tilde{\mu} (B_\rho(x)) \right)\geq \varepsilon_0\right \}
\end{equation*}
for  a positive constant  $\varepsilon_0$ (see Section 3 in
\cite{Scheven}).   Then, $ \tilde \nu$ is a $(m-4)-$rectifiable
measure and the singular set  $\tilde{\Sigma}$ is
$(m-4)-$rectifiable.
\end{thm}
An analogous rectifiable result on the concentration set of
stationary harmonic maps was established by Lin \cite {Lin}. A
similar result was also obtained by Tian \cite {Tian} for
Yang-Mills equations.

A difference between a sequence of stationary biharmonic maps
$\tilde u_i$ and the sequence $u_{\varepsilon_i}$ in Theorem 1.1
is that  $\tilde u_i$   converges smoothly to $\tilde u$ away from
the concentration set $\tilde \Sigma$, but it is hard to prove a
similar result for the sequence $u_{\varepsilon_i}$ in Theorem
1.1. As a consequence of this result for stationary harmonic maps,
the limiting defect measure $\tilde \nu$ is supported in the
energy concentration set $\tilde \Sigma$. However, for the
sequence $u_\varepsilon$ in Theorem 1.1, this is not obvious at
all. To overcome the difficulty,  only for $m=5$, we can adapt an
idea of Lin \cite{Acta} to prove that the $u_{\varepsilon_i}$
converges strongly in $u$ in $W^{2,2}(\Omega \backslash \Sigma_1
)$ (see below). More precisely, we have
\begin{thm}\label{thm:third} Let $u_{\varepsilon_i}$ be a
minimizer of $\mathbb H_{\varepsilon_i}$ in Theorem 1.1 and
\begin{equation*}
    \mu=\lim_{\varepsilon_i\to 0} \abs{\triangle u_{\varepsilon_i}}^2 dx =\abs{\triangle u}^2 dx +\nu,
\end{equation*}
for a measure $\nu\geq 0$. When $m=5$, we have:

    (1) There is a small positive constant $\varepsilon_1 <\varepsilon_0$  such that if
    \begin{equation*}
        \Sigma_1=\bigcap_{R>0}\left\{ x\in \Omega|\quad B_R(x)\subset \Omega,\, R^{-1}\mu(B_R(x))\geq \varepsilon_1\right\},
    \end{equation*}
    then $\Sigma_1$ is a relatively closed set of finite $1-$dimension Hausdorff measure and
    \begin{equation*}
        \Sigma_1=\mbox{spt}\, \nu\cup \mbox{sing}\, u.
    \end{equation*}

    (2) For $\mathcal H^1-$a.e. $x\in \Sigma_1$, $\nu=\Theta(x)\mathcal H^1\llcorner \Sigma_1$ and $\varepsilon_1\leq \Theta(x)\leq C(d(x,\partial \Omega))$, where $C(d(x,\partial\Omega))$ is a constant depending on the distance from $x$ to $\partial \Omega$.

    (3) The defect measure $\nu$ is $1$-rectifiable measure and hence $\Sigma_1$ is a $1$-rectifiable set.
\end{thm}

The paper is organized as follows. In Section 2,   we establish a
monotonicity and partial regularity of  the minimizer $
u_{\varepsilon}$ of $\Bbb H_{\varepsilon}$ in Theorem 1.1. In
Section 3, we prove the quantity $\Theta (x)$ exists for $\mathcal
H^{m-4}$ a.e. $x\in \Omega$ and give a proof of Theorem 1.2. In
Section 4, we prove a strong convergence of the sequence
$\{u_{\varepsilon}\}$ away from a concentration set and finally
complete a proof of Theorem 1.3.

\section{Perturbed variational problem and the partial regularity}\label{sec:var}

Let $F(u)$ be the relaxed energy defined  in Definition
\ref{defn:F}. It is easy to see that the minimum of the relaxed
energy $F(u)$ is achieved in $W^{2,2}_{u_0}(\Omega,S^n)$.
\begin{lem}\label{lem:21}
    There exists $\bar{u}\in W^{2,2}_{u_0}(\Omega,S^n)$ such that
    \begin{equation*}
        F(\bar{u})=\inf_{u\in W^{2,2}_{u_0}(\Omega,S^n)} F(u).
    \end{equation*}
\end{lem}
Let $u_i$ be a minimizing sequence of $F$. For each $u_i$, by definition, we can find a sequence of $u_{i,j}\in C_{u_0}^\infty$ such that $\lim_{j\to \infty} \mathbb H(u_{ij})$ can be arbitrarily close to $F(u_i)$. The proof of Lemma \ref{lem:21} follows from choosing a suitable $u_{i,k_i}$ for each $i$ and considering the weak limit of $u_{i,k_i}$.

However, we do not know how to prove that the minimizer given by
Lemma \ref{lem:21} is a biharmonic map. Instead, we start to
consider a perturbed functional $H_{\varepsilon}$ for $\varepsilon
>0$. The first observation is that
\begin{lem}\label{lem:approx}
    \begin{equation*}
        \inf_{W^{2,2}_{u_0}\cap C^0_{u_0}(\Omega,S^n)} {\mathbb H}(u)=\inf_{W^{2,2}_{u_0}\cap W^{1,m+1}(\Omega,S^n)} {\mathbb H}(u)=\inf_{C_{u_0}^\infty(\Omega,S^n)} {\mathbb H}(u).
    \end{equation*}
\end{lem}
\begin{proof}
    It is obvious that
    \begin{equation*}
        C^\infty_{u_0}(\Omega,S^n)\subset W^{2,2}_{u_0}\cap W^{1,m+1}(\Omega,S^n)\subset W^{2,2}_{u_0}\cap C^0_{u_0}(\Omega,S^n).
    \end{equation*}
    It suffices to show that for each $u\in W^{2,2}_{u_0}\cap C^0_{u_0}(\Omega,S^n)$, we can find a sequence of $u_k\in C^\infty_{u_0}(\Omega, S^n)$ such that
    \begin{equation*}
      \lim_{k\to \infty} \norm{u_k-u}_{W^{2,2}}=0.
    \end{equation*}
    For simplicity, let us assume $\Omega=B_1$. Define
    \begin{equation*}
     \left\{
     \begin{array}[]{ll}
       \tilde{u}=u-u_0 & \mbox{ for } x\in \overline{B_1} \\
       \tilde{u}=0 & \mbox{ for } x\in B_2\setminus \overline{B_1}.
     \end{array}
     \right.
    \end{equation*}
    Due to the boundary condition (\ref{eqn:boundary}), $\tilde{u}$ is in $W^{2,2}(B_2,\Real^{n+1})$. Let $\xi$ be a smooth function supported in $B_1(0)$ and satisfy
    \begin{equation*}
      \int_{\Real^m} \xi dx=1.
    \end{equation*}
    Set
    \begin{equation*}
      w_k(x)=\int_{\Real^m} k^m\xi(ky) \tilde{u}(x-y) dy
    \end{equation*}
    and
    \begin{equation*}
      \tilde{w}_k(x)=w_k( (1+\frac{2}{k}){x}).
    \end{equation*}
    By the definition of $\tilde{u}$ outside $B_1$ and the compact support of $\xi$, $\tilde{w}_k$ satisfies zero Dirichlet and Neumann boundary conditions on $\partial B_1$. It is obvious that
    \begin{equation*}
      \lim_{k\to \infty} \norm{\tilde{w}_k-\tilde{u}}_{W^{2,2}(B_1,\Real^{n+1})}=0.
    \end{equation*}
    We claim that $\tilde{w}_k$ converges to $\tilde{u}$ uniformly on $B_1$. In fact, $\tilde{u}( (1+\frac{2}{k})x)$ uniformly converges to $\tilde{u}(x)$ due to the uniform continuity of $u$ and $w_k(y)$ converges uniformly to $\tilde{u}(y)$ on $B_{3/2}$. We can now set
    \begin{equation*}
      u_k=\frac{\tilde{w}_k(x)+u_0}{\abs{\tilde{w}_k(x)+u_0}}.
    \end{equation*}
    It is straightforward to check that $u_k$ satisfies the boundary conditions (\ref{eqn:boundary}) and approaches $u$ in $W^{2,2}-$norm.
\end{proof}

As can be seen from the above proof in Lemma 2.2, we can
equivalently define $F(u)$ to be
    \begin{align*}
        F(u)=\inf \big\{ \liminf_{k\to \infty} {\mathbb H}(u_k) \,\,|\quad  \{u_k\}&\subset W^{2,2}_{u_0}\cap
        W^{1,m+1}(\Omega,S^n)\quad \mbox{and }\\
        &
         \quad u_k\wc u \mbox{ weakly in } W^{2,2}(\Omega,S^n)\, \big \}.
    \end{align*}

  The following observation plays an important role in this paper.
\begin{lem}\label{lem:hong} Let
    $u_\varepsilon$ be a minimizer of  $\mathbb H_\varepsilon$  in $W^{2,2}_{u_0}\cap W^{1,m+1}(\Omega,S^n)$.
    Then
    \begin{equation*}
        \lim_{\varepsilon\to 0} \int_{\Omega} \varepsilon \abs{\nabla u_\varepsilon}^{m+1} dx=0.
    \end{equation*}
\end{lem}
\begin{proof}
    Let $\varepsilon_i$ be any subsequence going to zero such that $\lim_{i\to \infty} \int_{\Omega} \varepsilon_i \abs{\nabla u_{\varepsilon_i}}^{m+1} dx$ exists.
    In the following, we write $u_i$ for $u_{\varepsilon_i}$ for
    simplicity. Using minimality of $u_i$, we have
    \begin{eqnarray*}
        & &\inf_{v\in W^{2,2}_{u_0}\cap C^0_{u_0}(\Omega,S^n)} {\mathbb H}(v)
       \leq\liminf_{i\to \infty} {\mathbb H}(u_i)\\
       &\leq&  \liminf_{i\to \infty}  {\mathbb H}(u_i)+\lim_{i\to \infty}\int_{\Omega} \varepsilon_i\abs{\nabla u_i}^{m+1} dx \leq \limsup_{i\to \infty} {\mathbb H}_{\varepsilon_i}(u_i)\\
        &\leq & \inf_{v\in W^{2,2}_{u_0}\cap W^{1,m+1}(\Omega,S^n)} \limsup_{i\to \infty}  {\mathbb H}_{\varepsilon_i}(v) \\
        &=& \inf_{v\in W^{2,2}_{u_0}\cap W^{1,m+1}(\Omega,S^n)} {\mathbb H}(v).
    \end{eqnarray*}
Using Lemma \ref{lem:approx}, we have
\[\lim_{i\to \infty} \int_{\Omega} \varepsilon_i \abs{\nabla u_{\varepsilon_i}}^{m+1} dx =0.\]
   This proves our claim.
\end{proof}

We can now prove the first part of Theorem \ref{thm:first}, namely,
\begin{prop}
    Let $u$ be a weak limit of $u_{\varepsilon_i}$ in $W^{2,2}$. Then $u$ is a minimizer of $F$ and $u$ is a biharmonic map.
\end{prop}
\begin{proof}
    By the definition of $F$ and Lemma \ref{lem:hong}, we have
    \begin{equation*}
        F(u)\leq \liminf_{i\to \infty} {\mathbb H}(u_{\varepsilon_i})=\inf_{v\in W^{2,2}_{u_0}\cap W^{1,m+1}(\Omega,S^n)} {\mathbb H}(v)=\inf_{v\in C^{\infty}_{u_0}(\Omega,S^n)}{\mathbb H}(v).
    \end{equation*}
    By the definition of $F$ again, $u$ is a minimizer of $F$ among all functions in $W^{2,2}_{u_0}(\Omega, S^n)$.

It is straightforward  to see that $u_\varepsilon$ satisfies the
Euler-Lagrange equation
\begin{eqnarray*}
    2\triangle^2 u_\varepsilon+ 2(\abs{\triangle u_\varepsilon}^2 +2\nabla\cdot(\nabla u_\varepsilon\cdot \triangle u_\varepsilon)-\triangle\abs{\nabla u_\varepsilon}^2)u_\varepsilon && \\
    -\varepsilon(m+1)[\nabla\cdot(\abs{\nabla u_\varepsilon}^{m-1}\nabla u_\varepsilon)+\abs{\nabla u_\varepsilon}^{m+1}u_\varepsilon]&=&0.
\end{eqnarray*}
This equation can be rewritten into a 'divergence' form (see
\cite{W1}) as follows,
\begin{eqnarray*}
    2\triangle(\nabla \cdot (\nabla u_\varepsilon\times u_\varepsilon))-4\nabla\cdot(\triangle u_\varepsilon\times \nabla u_\varepsilon) &&\\
    -\varepsilon(m+1)[\nabla\cdot (\abs{\nabla u_\varepsilon}^{m-1}\nabla u_\varepsilon\times u_\varepsilon)]&=&0.
\end{eqnarray*}
Due to Lemma \ref{lem:hong}, we conclude that the weak limit $u$ of $u_\varepsilon$ in $W^{2,2}_{u_0}(\Omega,S^n)$ satisfies
\begin{equation*}
    \triangle(\nabla \cdot (\nabla u\times u))-2\nabla\cdot(\triangle u\times \nabla u) =0.
\end{equation*}
Hence, $u$ is a biharmonic map (see \cite{W1}).
\end{proof}

The second part of Theorem \ref{thm:first} is to prove partial
regularity of the limiting map  $u$ of a sequence of minimizers
$\{u_{\varepsilon_i}\}$. It is well known that a monotonicity
formula plays an indispensable role in the proof of partial
regularity for stationary biharmonic maps. Since the minimizer
$u$ of $F$ is not stationary, we cannot prove a monotonicity
formula for $u$ directly. Fortunately, each $u_{\varepsilon}$  is
a minimizer  of ${\mathbb H}_{\varepsilon}$ in $W^{2,2}\cap
W^{1+m}(\Omega ; S^n)$. Hence, we will derive a monotonicity
formula for $u_\varepsilon$ first and then let $\varepsilon$ go to
zero.

Angelsberg \cite{Ang} gave a detailed derivation of a monotonicity
formula for stationary biharmonic maps. Since the functional
${\mathbb H}_\varepsilon$ is a perturbation of the Hessian energy,
most part of the proof in  \cite{Ang} can be used here. For the
convenience of readers, we stick to the notations used in
\cite{Ang} except for that we write subscripts of Greek letters to
indicate partial derivatives instead of Latin letters. For
example, $u_{\varepsilon,\alpha\beta}$ means
$\frac{\partial^2}{\partial x_\alpha \partial
x_\beta}u_{\varepsilon}$.

\begin{lem}\label{lem:pre}
    Let $u_\varepsilon$ be a minimizer of ${\mathbb H}_\varepsilon$ in $W^{2,2}\cap W^{1+m}(B_{2r},S^n)$.  Then we have
    \begin{eqnarray*}
        \int_{B_{2r}} -\nabla\cdot(\abs{\triangle u_\varepsilon}^2+\varepsilon \abs{\nabla u_\varepsilon}^{m+1}) +4 u_{\varepsilon,\gamma\gamma}u_{\varepsilon,\alpha\beta}\xi^\beta_\alpha &&\\
        +2u_{\varepsilon,\gamma\gamma}u_{\varepsilon,\beta}\xi^\beta_{\alpha\alpha}+\varepsilon(m+1)\abs{\nabla u_\varepsilon}^{m-1}u_{\varepsilon,\alpha} u_{\varepsilon,\beta} \xi^\beta_\alpha &=&0,
    \end{eqnarray*}
    for every test function $\xi\in C_0^\infty (B_{2r},\Real^m)$.
\end{lem}
The proof is just a direct computation (see \cite {CWY}). Now we
can state our monotonicity formula

\begin{thm}\label{thm:mono}
  Let $u_\varepsilon$ be a minimizer of ${\mathbb H}_\varepsilon$ on $B_{R_0}$ for some $R_0>0$. Then for all $\rho$ and $r$ with $0<\rho<r<R_0/2$, we have
  \begin{equation*}
      r^{4-m}\int_{B_{r}}\abs{\triangle u_\varepsilon}^2+\varepsilon\abs{\nabla u_\varepsilon}^{m+1} dx -\rho^{4-m}\int_{B_{\rho}}\abs{\triangle u_{\varepsilon}}^2+\varepsilon \abs{\nabla u_\varepsilon}^{m+1} dx =P+R+Q,
  \end{equation*}
  where
\begin{eqnarray*}
    P&=&4\int_{B_{r}\setminus B_{\rho}}\left( \frac{(u_{\varepsilon,\beta}+x^\alpha u_{\varepsilon,\alpha\beta})^2}{\abs{x}^{m-2}}
    +\frac{(m-2)(x^\alpha u_{\varepsilon,\alpha})^2}{\abs{x}^m} \right)dx, \\
    &&+\varepsilon(m+1)\int_{B_{r}\setminus B_{\rho}} \frac{\abs{\nabla u_\varepsilon}^{m-1}(x^\alpha u_{\varepsilon,\alpha})^2}{\abs{x}^{m-2}} dx, \\
    R&=& 2\int_{\partial B_{r}\setminus \partial B_{\rho}}\left( -\frac{x^\alpha u_{\varepsilon,\beta} u_{\varepsilon,\alpha\beta}}{\abs{x}^{m-3}}+2\frac{(x^\alpha u_{\varepsilon,\alpha})^2}{\abs{x}^{m-1}}-2\frac{\abs{\nabla u_\varepsilon}^2}{\abs{x}^{m-3}}  \right)d\sigma \\
    Q&=& \varepsilon(3-m)\int_\rho^r \tau^{3-m}\int_{B_{\tau}} \abs{\nabla u_\varepsilon}^{m+1} dxd\tau.
\end{eqnarray*}
\end{thm}

\begin{proof}
    We follow the proof in \cite{Ang}. Choose a test function  $\xi(x)=\psi^t(\frac{\abs{x}}{\tau})x$,
    where $\psi=\psi^t:\Real_+\to [0,1]$ is smooth with compact support on $[0,1]$ and $\psi^t \equiv 1$ on $[0,1-t]$. Then by Lemma
    \ref{lem:pre}, we have
    \begin{eqnarray*}\label{eqn:1}
        0&=&\int_{\Real^m} \left( (4-m) \abs{\triangle u_\varepsilon}^2\psi -\abs{\triangle u_\varepsilon}^2 \psi_\alpha x^\alpha +4u_{\varepsilon,\alpha\alpha}u_{\varepsilon,\beta\gamma}\psi_\beta x^\gamma \right.\\
        &&\left. +4u_{\varepsilon,\alpha\alpha}u_{\varepsilon,\beta}\psi_\beta +2u_{\varepsilon,\alpha\alpha}u_{\varepsilon,\beta}\psi_{\gamma\gamma}x^\beta \right) \\
    &&+\varepsilon (\psi\abs{\nabla u_\varepsilon}^{m+1}-\abs{\nabla u_\varepsilon}^{m+1}\psi_\alpha x^\alpha +(m+1)\abs{\nabla u_\varepsilon}^{m-1}u_{\varepsilon,\alpha}\psi_\alpha u_{\varepsilon,\beta} x^\beta) dx.
    \end{eqnarray*}
    Since
    $\psi_\alpha(\frac{\abs{x}}{\tau})=\frac{1}{\tau}\psi'(\frac{\abs{x}}{\tau})\frac{x^\alpha}{\abs{x}}$, we
    have
    \begin{equation}\label{eqn:2}
        0=\int_{\Real^m} \cdots+\varepsilon (\psi\abs{\nabla u_\varepsilon}^{m+1}-\abs{\nabla u_\varepsilon}^{m+1}\frac{1}{\tau}\psi'\abs{x}
        +(m+1)\abs{\nabla u_\varepsilon}^{m-1}\frac{1}{\tau}\psi'\frac{(u_{\varepsilon,\beta} x^\beta)^2}{\abs{x}})
        dx,
    \end{equation}
    where for simplicity we use `$\cdots$' to denote those terms which are the same as in \cite{Ang}.

    Set
    \begin{equation*}
        I^t(\tau)=\tau^{4-m}\int_{\Real^m}(\abs{\triangle u_\varepsilon}^2+\varepsilon\abs{\nabla u_\varepsilon}^{m+1})\psi^{t}(\frac{\abs{x}}{\tau})
        dx.
    \end{equation*}
    We have
    \begin{eqnarray*}
        \tau^{m-3}\frac{d}{d\tau}I^t(\tau)&=& (4-m)\int_{\Real^m} (\abs{\triangle u_\varepsilon}^2+\varepsilon \abs{\nabla u_\varepsilon}^{m+1})\psi dx \\
        &&-\frac{1}{\tau}\int_{\Real^m}(\abs{\triangle u_\varepsilon}^2 +\varepsilon \abs{\nabla u_\varepsilon}^{m+1})\psi'\abs{x}dx \\
        &=& \cdots \\
        &&+\varepsilon[\int_{\Real^m} \abs{\nabla u_\varepsilon}^{m+1}\psi dx -\frac{1}{\tau}\int_{\Real^m}\abs{\nabla u_\varepsilon}^{m+1}\psi' \abs{x}  dx ]\\
        && +\varepsilon(3-m)\int_{\Real^m} \abs{\nabla u_\varepsilon}^{m+1}\psi dx \\
        &=& \cdots +\varepsilon\int_{\Real^m}-(m+1)\abs{\nabla u_\varepsilon}^{m-1}\psi'\frac{1}{\tau}\frac{(x^\alpha u_{\varepsilon,\alpha})^2}{\abs{x}} dx \\
        && + \varepsilon(3-m)\int_{\Real^m}\abs{\nabla u_\varepsilon}^{m+1}\psi dx.
    \end{eqnarray*}
    Here we have used equation (\ref{eqn:2}) in the last equality.
 Multiplying both sides by $\tau^{3-m}$ and integrating over $\tau$ from $\rho$ to $r$ yield
    \begin{eqnarray*}
        I^t(r)-I^t(\rho) &=& \cdots + \varepsilon \int_\rho^r \tau^{2-m}\int_{\Real^m} -(m+1)\abs{\nabla u_\varepsilon}^{m-1}\psi' \frac{(x^\alpha u_{\varepsilon,\alpha})^2}{\abs{x}} dxd\tau \\
        && +\varepsilon(3-m) \int_\rho^r \tau^{3-m}\int_{\Real^m} \abs{\nabla u_\varepsilon}^{m+1}\psi dx d\tau.
    \end{eqnarray*}
    Letting $t$ go to zero and applying Lemma 2 in the Appendix of \cite{Ang}, we obtain
    \begin{eqnarray*}
        && r^{4-m}\int_{B_{r}}\abs{\triangle u_\varepsilon}^2+\varepsilon\abs{\nabla u_\varepsilon}^{m+1}dx -\rho^{4-m}\int_{B_{\rho}}\abs{\triangle u_\varepsilon}^2+\varepsilon \abs{\nabla u_\varepsilon}^{m+1} dx\\
    &=& \int_{B_r\setminus B_\rho} \left( 4\frac{u_{\varepsilon,\alpha\alpha}u_{\varepsilon,\beta\gamma}x^\beta x^\gamma}{\abs{x}^{m-2}}+ 8\frac{u_{\varepsilon,\alpha\alpha}u_{\varepsilon,\beta} x^\beta}{\abs{x}^{m-2}} \right)dx -2\int_{\partial B_r\setminus \partial B_\rho}\frac{u_{\varepsilon,\alpha\alpha}u_{\varepsilon,\beta} x^\beta}{\abs{x}^{m-3}}d\sigma \\
    &&+ \varepsilon(m+1) \int_{B_{r}\setminus B_{\rho}}\frac{\abs{\nabla u_\varepsilon}^{m-1}(x^\alpha u_{\varepsilon,\alpha})^2}{\abs{x}^{m-2}} dx \\
        && +\varepsilon(3-m)\int_\rho^r \tau^{3-m}\int_{B_{\tau}}\abs{\nabla u_\varepsilon}^{m+1} dx d\tau.
    \end{eqnarray*}
    For the first line in the right hand side of the above equation, it needs further transformations before reaching the final form appeared in the statement of the theorem as given in \cite{Ang}. However, this does not concern us, since the last two terms above are in their final form.
\end{proof}
\begin{rem}
    If we compare Theorem \ref{thm:mono} with the monotonicity formula of the biharmonic maps
    in \cite{CWY} and \cite{Ang}, there is an additional term in $P$ and a new term $Q$.
    The additional term in $P$ is the contribution of $\varepsilon\abs{\nabla u}^{m+1}$ term in the perturbed energy.
    The new term $Q$ is caused by the fact that the two terms in the perturbed energy transform differently when scaled.
    Moreover, $Q$ is of the unfavorable sign and we need to get rid of it by taking $\varepsilon$ to zero.
\end{rem}

Let $u_{\varepsilon_i}$ be the sequence in the statement of Theorem \ref{thm:first}.
Due to the minimizing property of $u_{\varepsilon_i}$,
\begin{equation*}
    {\mathbb H}(u_{\varepsilon_i})\leq {\mathbb H}_{\varepsilon_i}(u_{\varepsilon_i}) \leq
  C
\end{equation*}
for some constant $C>0$ independent of $i$. Set
\begin{equation*}
    \Sigma=\bigcap_{R>0}\left\{ x_0\in \Omega| B_R(x_0)\subset \Omega,\quad \liminf_{i\to \infty} R^{4-m}\int_{B_R(x_0)} \abs{\triangle u_{\varepsilon_i}}^2 dx \geq \varepsilon_0 \right\}
\end{equation*}
for a sufficiently small constant $\varepsilon_0$ to be fixed
later. For the proof of ${\mathcal H}^{m-4}(\Sigma)< +\infty$, we
refer to the proof of Theorem 3.4 in \cite{Scheven}. For the
relative closeness of $\Sigma$,  an elementary proof will be given
in the last section in the proof of Theorem \ref{thm:third} (see also
\cite{GHY}).

Now we prove Theorem \ref{thm:first},
\begin{proof}[Proof of Theorem \ref{thm:first}]
    The first part is already proved. It suffices to prove the partial regularity.
  Let $x$ be a point in $\Omega\setminus \Sigma$. Without loss of generality, we assume that $x$ is the origin. By the definition of $\Sigma$, there exists some $R>0$ such that $B_{R}\subset \Omega$ and (taking a subsequence if necessary)
\begin{equation*}
  \lim_{i\to \infty}R^{4-m}\int_{B_{R}}\abs{\triangle u_{\varepsilon_i}}^2 dx<\varepsilon_0.
\end{equation*}
For simplicity, we will write $u_i$ for $u_{\varepsilon_i}$. It is
easy to see that for each $y\in B_{R/2}$,
\begin{equation*}
    \lim_{i\to \infty} (R/2)^{4-m}\int_{B_{R/2}(y)}\abs{\triangle u_i}^2 dx < C(m)\varepsilon_0.
\end{equation*}

\noindent{\bf We claim:} For almost every $y\in B_{R/2}$, for any
$r<R/8$, there exists some radius $r/2<\rho<r$ such that
\begin{equation*}
  \rho^{4-m}\int_{B_\rho(y)} \abs{\triangle u}^2 dx\leq C(m)\varepsilon_0.
\end{equation*}

Before we prove this claim, we show how Theorem \ref{thm:first} follows
from this claim. Since $u$ takes value in the sphere, it is
obvious that
\begin{equation*}
  \rho^{4-m}\int_{B_\rho(y)} \abs{\triangle u}^2 +\abs{\nabla u}^4 dx \leq C(m) \varepsilon_0.
\end{equation*}
This implies that
\begin{equation*}
    (r/2)^{4-m}\int_{B_{r/2}(y)}\abs{\triangle u}^2 +\abs{\nabla u}^4 dx\leq C(m) \varepsilon_0.
\end{equation*}
By the arbitrariness of $r$ and the density of $y$, we obtain
\begin{equation*}
  \norm{\nabla^2 u}_{L^{2,m-2}(B_{R/3})}+\norm{\nabla u}_{L^{4,m-4}(B_{R/3})}\leq C(m)\varepsilon_0.
\end{equation*}
Here $L^{p,m-p}$ is the standard Morrey space (see \cite{Struwe}).
For $\varepsilon_0$ sufficiently small,  $u$ is smooth in
$B_{R/3}$ since $u$ is biharmonic (cf. \cite{Struwe}).

Now let us prove the Claim. Without loss of generality, we assume
that $R=2$ and $y$ is the origin.

Since
\begin{equation}\label{eqn:A}
  \int_{B_{1}\setminus B_{1/2}} \abs{\triangle u_i}^2 dx< C(m) \varepsilon_0,
\end{equation}
we can choose $r$ such that
\begin{equation}\label{eqn:B}
  \int_{\partial B_{r}} \abs{\triangle u_i}^2 dx\leq C(m) \varepsilon_0,
\end{equation}
for infinitely many $i$'s. We assume by taking subsequence that this is true for all $i$. Assume without loss of generality that $r=1$. (Otherwise, consider $B_{r}$ instead of $B_{1}$.)

Following Struwe \cite{Struwe}, we write the monotonicity formula in the following form.

\begin{eqnarray}\label{eqn:kaka}
    \sigma_i(r)-\sigma_i(\rho)&=&\int_{B_{r}\setminus B_{\rho}} \left( \frac{\abs{u_{i,\beta}+x^\alpha u_{i,\alpha\beta}}^2}{\abs{x}^{m-2}}+(m-2)\frac{\abs{x^\alpha u_{i,\alpha}}^2}{\abs{x}^m} \right) \\ \nonumber
    && +\varepsilon_i(m+1)\int_{B_{r}\setminus B_{\rho}}\frac{\abs{\nabla u_i}^{m-1}(x^\alpha u_{i,\alpha})^2}{\abs{x}^{m-2}} \\ \nonumber
    && +\varepsilon_i(3-m)\int_\rho^r \tau^{3-m}\int_{B_{r}}\abs{\nabla u_i}^{m+1}dx d\tau,
\end{eqnarray}
where $\sigma_i(r)=\sigma_{i,1}(r)+\sigma_{i,2}(r)$ with
\begin{equation*}
    \sigma_{i,1}(r)=r^{4-m}\int_{B_{r}} \abs{\triangle u_i}^2 +\varepsilon_i \abs{\nabla u_i}^{m+1} d\sigma
\end{equation*}
and
\begin{equation*}
    \sigma_{i,2}(r)=r^{3-m}\int_{\partial B_{r}} (2x^\alpha u_{i,\alpha\beta}u_{i,\beta} +4\abs{\nabla u_i}^2 -4r^{-2}\abs{x^\alpha u_{i,\alpha}}^2) d\sigma.
\end{equation*}

Denote by $\mathcal R(\varepsilon_i,\rho)$ the last term in
equation (\ref{eqn:kaka}).  Let $E_1$ be the intersection of the
sets of Lebesgue points of $\abs{\nabla u_i}^{m+1}$ for all $i$.
Then the complement of $E_1$ is a set of zero Lebesgue measure.
For each $y\in E_1$,
\begin{equation*}
    \lim_{\rho\to 0}\mathcal R(\varepsilon_i,\rho)
\end{equation*}
exists. In the following, we assume $y\in E_1$ and denote the
above limit by $\mathcal R(\varepsilon_i,0)$.

For fixed $i$ and any $k\in \mathbb N$, there is a good slice $0<r_k<\frac{1}{k}$ such that
\begin{eqnarray*}
    \abs{\sigma_i(r_k)}&\leq& Cr_k^{4-m}\int_{B_{r_k}} \abs{\triangle u_i}^2+\varepsilon_i \abs{\nabla u_i}^{m+1} dx
     +Cr_k^{5-m}\int_{\partial B_{r_k}} (\abs{\nabla^2 u_i}^2+r_k^{-2}\abs{\nabla u_i}^2) d\sigma \\
    &\leq& Cr_k^{4-m}\int_{B_{2r_k}}(\abs{\nabla^2 u_i}^2+\varepsilon_i \abs{\nabla u_i}^{m+1}+r_k^{-2}\abs{\nabla u_i}^2) dx.
\end{eqnarray*}
Let $E_2$ be the intersection of the sets of Lebesgue points of
$\abs{\nabla^2 u_i}^2+\abs{\nabla u_i}^2$ for all $i$. The
complement of $E_2$ is also of Lebesgue measure zero. If we assume
$y\in E_1\cap E_2$, we have
\begin{equation}\label{eqn:C}
    \lim_{k\to \infty} \abs{\sigma_i(r_k)}=0.
\end{equation}
In (\ref{eqn:kaka}), set $\rho=r_k$ and let $k$ go to infinity.
Then we obtain
\begin{equation*}
    \sigma_{i}(1)-\lim_{k\to \infty} \sigma_i(r_k)=\int_{B_{1}}\left( \cdots \right) +\mathcal R(\varepsilon_i,0),
\end{equation*}
where `$\cdots$' stands for the two positive integrals in
(\ref{eqn:kaka}).

By (\ref{eqn:B}), we know
\begin{equation*}
    \sigma_i(1)\leq C(m)\varepsilon_0.
\end{equation*}
Therefore, we prove
\begin{equation}\label{eqn:haha}
    \int_{B_{1}} \left( \frac{\abs{u_{i,\beta}+x^\alpha u_{i,\alpha\beta}}^2}{\abs{x}^{m-2}}
    +(m-2)\frac{\abs{x^\alpha u_{i,\alpha}}^2}{\abs{x}^m} \right)\leq C(m)\varepsilon_0 -\mathcal R(\varepsilon_i,0).
\end{equation}

Since we know
\begin{equation*}
    \varepsilon_i\int_{B_{1}}\abs{\nabla u_i}^{m+1}dx\to 0,
\end{equation*}
we may assume that there is a subsequence of $i$ (still denoted by $i$) such that
\begin{equation*}
    T(x)=\sum_{i=1}^\infty 2^{i} \varepsilon_i \abs{\nabla u_{\varepsilon_i}}^{m+1}\in L^1(B_{1}).
\end{equation*}
Hence,
\begin{eqnarray*}
    \mathcal R(\varepsilon_i,0)=\varepsilon_i \int_0^1 \tau^{3-m}\int_{B_{r}} \abs{\nabla u_i}^{m+1}dx d\tau &\leq
    & \frac{1}{2^i}\int_0^1 \tau^{3-m}\int_{B_{r}} T(x)dx d\tau.
\end{eqnarray*}

Let $E_3$ be the set of Lebesgue points of $T$. If $y\in E_1\cap E_2\cap E_3$, then
\begin{equation}\label{eqn:R0}
    \lim_{i\to \infty} \mathcal R(\varepsilon_i,0)=0.
\end{equation}
With these preparations, we can now estimate $\sigma_{i,2}$ from below. By (\ref{eqn:haha}),
for any $r>0$,
\begin{eqnarray*}
    &&\inf_{r/2<\rho<r} \rho^{3-m}\int_{\partial B_{\rho}}(\abs{u_{i,\beta}+x^\alpha u_{i,\alpha\beta}}^2+4\rho^{-2}\abs{x^\alpha u_{i,\alpha}}^2) d\sigma \\
    &\leq& C\int_{B_{r}\setminus B_{r/2}}\left( \frac{\abs{u_{i,\beta}+x^\alpha u_{i,\alpha\beta}}^2}{\abs{x}^{m-2}}+(m-2)\frac{\abs{x^\alpha u_{i,\alpha}}^2}{\abs{x}^m}
     \right)dx\\
    &\leq& C(m)\varepsilon_0-\mathcal R(\varepsilon_i,0).
\end{eqnarray*}

Estimating
\begin{equation*}
    2x^\alpha u_{i,\alpha\beta}u_{i,\beta}+4\abs{\nabla u_i}^2=2(u_{i,\beta}+x^\alpha u_{i,\alpha\beta})u_{i,\beta}+2\abs{\nabla u_i}^2\geq -\abs{u_{i,\beta}+x^\alpha u_{i,\alpha\beta}}^2,
\end{equation*}
we bound
\begin{equation*}
    \sigma_{i,2}(\rho)\geq -\rho^{3-m}\int_{\partial B_{\rho}}(\abs{u_{i,\beta}+x^\alpha u_{i,\alpha\beta}^2}+4\rho^{-2}\abs{x^\alpha u_{i,\alpha}}^2)d\sigma.
\end{equation*}
Therefore,
\begin{equation*}
    \sup_{r/2<\rho<r}\sigma_{i,2}(\rho)\geq -C(m)\varepsilon_0+\mathcal R(\varepsilon_i,0).
\end{equation*}
Now from the monotonicity formula, for a suitable radius in $(r/2,r)$,
\begin{eqnarray*}
    \sigma_{i,1}(\rho)&\leq& \sigma_{i}(\rho)-\sigma_{i,2}(\rho)\\
    &\leq& \sigma_{i}(1)-\mathcal R(\varepsilon_i,\rho) +C(m)\varepsilon_0 -\mathcal R(\varepsilon_i,0).
\end{eqnarray*}
Noticing (\ref{eqn:R0}) and the fact that $\lim_{i\to \infty} \mathcal R(\varepsilon_i,\rho)=0$,  we have by letting $i\to \infty$
\begin{equation*}
    \rho^{4-m}\int_{B_{\rho}}\abs{\triangle u}^2 dx\leq C(m)\varepsilon_0.
\end{equation*}
Thus, we finish the proof of the Claim for $y\in E_1\cap E_2\cap
E_3$. Since the Lebesgue measure of the complement of $E_1\cap
E_2\cap E_3$ is zero, we prove our claim.
\end{proof}

\section{Further results on the monotonicity formula}\label{sec:further}

In this section, let  $\sigma_{i,1}(r)$, $\sigma_{i,2}(r)$ and
$\sigma_i(r)$ be defined in Section 2. For simplicity, we denote

\[\sigma_1(r):=\liminf_{i\to \infty}
\sigma_{i,1}(r),\quad \sigma_2(r)=\liminf_{i\to
\infty}\sigma_{i,2}(r) ,\quad  \sigma(r)  :=\liminf_{i\to \infty}
\sigma_i(r).
 \]
The main purpose of this section is to show that
for $\mathcal H^{n-4}-$a.e. $y\in \Omega$, the limit $\lim_{r\to
0}\sigma_1(r)$ exists. We will use this result to show that the
defect measure is rectifiable.

The following is a lemma which will be used many times in this
section. Although it may be well known,  we would like to give a
proof here for the
 completeness.

\begin{lem}
Let $f_i$ be a sequence of nonnegative integrable functions on
$B_1$. For each $r_1,r_2>0$, there exists a constant $\rho\in
[r_1,r_2]$ such that
\begin{equation*}
    (r_2-r_1)\liminf _{i\to \infty} \int_{\partial B_\rho} f_i dx\leq 2 \liminf_{i\to \infty} \int_{B_{r_2}} f_i dx.
\end{equation*}
\end{lem}
\begin{proof}
    If otherwise, then
    \begin{equation*}
        \liminf_{i\to \infty} \int_{\partial B_\rho} f_idx > \frac{2}{r_2-r_1}\liminf_{i\to \infty}\int_{B_{r_2}}f_i
        dx
    \end{equation*}
    for almost all $\rho\in [r_1,r_2]$.
    Integrating both sides in $\rho$ over $[r_1,r_2]$, we obtain
    \begin{equation*}
        \int_{r_1}^{r_2}\liminf_{i\to \infty}\int_{\partial B_\rho} f_i dx > 2\liminf_{i\to \infty} \int_{B_{r_2}}f_i dx.
    \end{equation*}
On the other hand,  by  Fatou's lemma, we have
    \begin{equation*}
        \int_{r_1}^{r_2}\liminf_{i\to \infty}\int_{\partial B_\rho} f_i dx \leq \liminf_{i\to \infty} \int_{B_{r_2}\setminus B_{r_1}} f_i dx.
    \end{equation*}
    This is a contradiction.
\end{proof}

\begin{lem}\label{lem:41} Suppose that $K$ is a compact set in $\Omega$ and $d$ is the distance from $K$ to $\partial \Omega$. For every $x\in K$,
    \begin{equation*}
        \sigma_1(r)\leq C(d),
    \end{equation*}
    when $r<d/10$.
\end{lem}
\begin{proof}
    Assume that $x$ is the origin. Since the total energy is bounded,
    \begin{equation*}
        \sigma_1(d/2)\leq C(d).
    \end{equation*}
    As in the proof of Theorem \ref{thm:first}, there exists $\tilde{r}\in [d/4,d/2]$ such that
    \begin{equation}\label{tilder}
        \sigma(\tilde{r})\leq C(d).
    \end{equation}
For each $r$ wiith $0<r<d/8$,  using Lemma 3.1, there exists
$\rho_r\in [r,3r/2]$ such that
\begin{eqnarray}\label{eqn:theone}
    & &\abs{\sigma_2(\rho_r) }
    \leq   \liminf_{i\to \infty}C \rho_r^{3-m}\int_{\partial B_{\rho_r}} \rho_r \abs{\nabla u_i}\abs{\nabla^2 u_i} + \abs{\nabla u_i}^2 d\sigma \\
    \nonumber
&\leq& \liminf _{i\to \infty}Cr^{2-m}\int_{B_{3r/2}} r\abs{\nabla u_i}\abs{\nabla^2 u_i}+\abs{\nabla u_i}^2 dx  \\ \nonumber
    &\leq& \liminf_{i\to \infty} \eta r^{4-m} \int_{B_{3r/2}}\abs{\nabla^2 u_i}^2 dx +C(\eta) r^{2-m}\int_{B_{3r/2}} \abs{\nabla u_i}^2
    dx
\end{eqnarray}
for a constant $\eta$ which will be fixed later. Set
\begin{equation*}
    \delta(r)=\sigma_1(r)+\liminf_{i\to \infty}r^{2-m}\int_{B_r}\abs{\nabla u_i}^2 dx.
\end{equation*}
By an interpolation inequality of Nirenberg \cite{Nir}, we have
\begin{eqnarray}\label{eqn:goodpoint}
    &&\liminf_{i\to \infty}r^{2-m}\int_{B_{3r/2}}\abs{\nabla u_i}^2 dx \\ \nonumber
    &\leq&  \liminf_{i\to \infty} C \left( r^{4-m}\int_{B_{3r/2}} \abs{\nabla u_i}^4 dx \right)^{1/2}   \\ \nonumber
    &\leq& \liminf_{i\to \infty} C \norm{u_i}_{L^\infty}\left( r^{4-m}\int_{B_{3r/2}} \abs{\nabla^2 u_i}^2 dx \right)^{1/2} +C\norm{u_i}_{L^\infty}^2 \\ \nonumber
    &\leq& C (\delta (2r))^{1/2}+ C.
\end{eqnarray}
Letting $i$ go to infinity in the monotonicity formula
(\ref{eqn:kaka}), we obtain
\begin{equation*}
    \sigma_1(\rho_r)+\sigma_2(\rho_r)\leq \sigma(\tilde{r}).
\end{equation*}
Hence,
\begin{eqnarray*}
    \delta(r) &\leq & \sigma_1(r)+\liminf_{i\to \infty}r^{2-m} \int_{B_r}\abs{\nabla u_i}^2 dx \\
    &\leq& C\sigma_1(\rho_r) +\liminf_{i\to \infty}r^{2-m} \int_{B_r}\abs{\nabla u_i}^2 dx\\
    &\leq& C\sigma(\tilde{r})+C\abs{\sigma_2(\rho_r)} +\liminf_{i\to \infty} r^{2-m}\int_{B_r}\abs{\nabla u_i}^2 dx\\
    &\leq& C\sigma(\tilde{r})+\liminf_{i\to \infty} C \eta r^{4-m}\int_{B_{3r/2}} \abs{\nabla^2 u_i}^2 dx +C(\eta) r^{2-m}\int_{B_{3r/2}} \abs{\nabla u_i}^2 dx \\
    &\leq& C\sigma(\tilde{r})+C\eta\delta(2r) +\liminf_{i\to \infty} C(\eta) r^{2-m}\int_{B_{3r/2}} \abs{\nabla u_i}^2 dx.
\end{eqnarray*}
 By choosing $\eta$ sufficiently small, we have
\begin{eqnarray*}
    \delta(r)
    &\leq& C\sigma(\tilde{r})+ \frac{1}{2} \sigma_1(2r) +C\liminf_{i\to \infty} r^{2-m}\int_{B_{2r}} \abs{\nabla u_i}^2 dx \\
    &\leq& \frac{1}{2}\delta(2r)+C\delta(2r)^{1/2} +C(d) \\
    &\leq& \frac{3}{4}\delta(2r)+C(d).
\end{eqnarray*}
An  iteration argument yields
\begin{equation*}
    \sigma_1(r)\leq \delta(r)\leq C(d).
\end{equation*}
\end{proof}

Set
\begin{equation*}
    E=\left\{x_0\in \Omega|\quad \limsup_{r\to 0} r^{4-m}\int_{B_r(x_0)} \abs{\nabla u}^4 dx > 0\right\}.
\end{equation*}
By Corollary 3.2.3 in \cite{Ziemer}, $\mathcal H^{m-4}(E)=0$. From now on, pick $y\notin E$ and assume without loss of generality it is the origin.
\begin{lem}\label{lem:42}
    For $y\notin E$, the limit
    \begin{equation*}
        \lim_{r\to 0} \sigma(r)
    \end{equation*}
    exists and is nonnegative.
\end{lem}
\begin{proof}
    By (\ref{eqn:kaka}), $\sigma(r)$ is non-increasing (as $r\to 0$), so it suffices to show that for some sequence of $\rho_k$ going to zero, $\sigma(\rho_k)$ has a lower bound. Take any sequence $r_k$ going to zero. For each $r_k$, there is a good radius $\rho_k\in [r_k,2r_k]$ such that as in (\ref{eqn:theone})
\begin{eqnarray*}\label{eqn:thetwo}
    && \abs{\sigma_2(\rho_k) } \leq  \liminf_{i\to \infty}C \rho_k^{3-m}\int_{\partial B_{\rho_k}} \rho_k \abs{\nabla u_i}\abs{\nabla^2 u_i} + \abs{\nabla u_i}^2 d\sigma \\
    &\leq& \liminf _{i\to \infty}Cr_k^{2-m}\int_{B_{2r_k}} r_k\abs{\nabla u_i}\abs{\nabla^2 u_i}+\abs{\nabla u_i}^2 dx  \\
    &\leq& C \left( \liminf_{i\to \infty} r_k^{4-m}\int_{B_{2r_k}}\abs{\nabla^2 u_i}^2 dx \right)^{1/2}\left( \liminf_{i\to \infty} r_k^{2-m}\int_{B_{2r_k}}\abs{\nabla u_i}^2 dx \right)^{1/2} \\
    && + C \liminf_{i\to \infty} r_k^{2-m}\int_{B_{2r_k}}\abs{\nabla u_i}^2 dx.
\end{eqnarray*}
By our choice of $y$, we note
\begin{equation}\label{eqn:choice}
    \lim_{k\to \infty}\liminf_{i\to \infty} r_k^{2-m}\int_{B_{2r_k}}\abs{\nabla u_i}^2 dx=\lim_{k\to \infty}r^{2-m}_k \int_{B_{2r_k}} \abs{\nabla u}^2 dx=0.
\end{equation}
Combing this with Lemma \ref{lem:41} yields
\begin{equation*}
    \lim_{k\to \infty} \sigma_2(\rho_k)=0.
\end{equation*}
Due to the monotonicity of $\sigma(r)$,
\begin{equation*}
    \lim_{r\to 0} \sigma(r) =\lim_{k\to \infty} \sigma(\rho_k)=\lim_{k\to \infty} \sigma_1(\rho_k)\geq 0.
\end{equation*}
\end{proof}

The following theorem is the main result of this section.
\begin{thm}\label{thm:15}
    For all $y\notin E$, the limit
    \begin{equation*}
        \lim_{r\to 0} \sigma_1(r)
    \end{equation*}
    exists and
    \begin{equation*}
        \lim_{r\to 0} \sigma_1(r)=\lim_{r\to 0} \sigma(r).
    \end{equation*}
\end{thm}

\begin{proof}
    It suffices to show that for any sequence $r_k$ going to zero,
    \begin{equation*}
        \lim_{k\to \infty}\sigma_1(r_k)=\lim_{r\to 0}\sigma(r).
    \end{equation*}
    Let $\theta_k$ be a sequence of positive number in $(0,1/2)$ to be determined later. Using (3.4) and Lemma 3.1,  there exists
     a constant $\rho_k\in [r_k,r_k(1+\theta_k)]$ such that
\begin{eqnarray*}
    && \abs{\sigma_2(\rho_k) }\\
    &\leq & \liminf_{i\to \infty}C \rho_k^{3-m}\int_{\partial B_{\rho_k}} \rho_k \abs{\nabla u_i}\abs{\nabla^2 u_i} + \abs{\nabla u_i}^2 d\sigma \\
    &\leq& \liminf_{i\to \infty} C\theta_k^{-1} r_k^{2-m}\int_{B_{(1+\theta_k)r_k}} r_k\abs{\nabla u_i}\abs{\nabla^2 u_i}+\abs{\nabla u_i}^2 dx   \\
    &\leq& C \left( \liminf_{i\to \infty} r_k^{4-m}\int_{B_{2r_k}}\abs{\nabla^2 u_i}^2 dx \right)^{1/2}\left( \theta_k^{-2}\liminf_{i\to \infty} r_k^{2-m}\int_{B_{2r_k}}\abs{\nabla u_i}^2 dx \right)^{1/2} \\
    && + C \theta_k^{-1}\liminf_{i\to \infty} r_k^{2-m}\int_{B_{2r_k}}\abs{\nabla u_i}^2 dx  .
\end{eqnarray*}
Since $y\notin E$, which implies that (\ref{eqn:choice}) is true,
we can choose $\theta_k$ going to zero so that
\begin{equation*}
    \lim_{k\to \infty} \sigma_2(\rho_k)=0.
\end{equation*}
As in Lemma \ref{lem:42}, we see
\begin{equation*}
    \lim_{k\to \infty} \sigma_1(\rho_k)=\lim_{r\to 0} \sigma(r).
\end{equation*}
By the same reason, we can find a sequence $\rho'_k \in [r_k(1-\theta_k),r_k]$ such that
\begin{equation*}
    \lim_{k\to \infty} \sigma_1(\rho'_k)=\lim_{r\to 0} \sigma(r).
\end{equation*}
However,
\begin{align*}
   (\frac{r_k}{\rho_k'})^{4-m}(\rho_k')^{4-m}\int_{B_{\rho_k'}} \abs{\triangle u_i}^2 dx&\leq r_k^{4-m}\int_{B_{r_k}} \abs{\triangle u_i}^2 dx\\
   &\leq (\frac{r}{\rho_k})^{4-m}\rho_k^{4-m}\int_{B_{\rho_k}} \abs{\triangle u_i}^2
   dx.
\end{align*}
Taking the limit of $i$ going to infinity and then $k$ going to
infinity, we obtain
\begin{equation*}
    \lim_{k\to \infty} \sigma_1(\rho_k')\leq \lim_{k\to \infty}\sigma_1(r_k) \leq \lim_{k\to \infty} \sigma_1(\rho_k).
\end{equation*}
This proves Theorem 3.1.
\end{proof}

One can see from the above  proofs that the same argument works
for a sequence of stationary biharmonic maps. In the following, we
use this observation to prove Theorem \ref{thm:second}.
\begin{proof}[Proof of Theorem  \ref{thm:second}]
    Let $\tilde{u}_i:\Omega\to N$ be a sequence of stationary biharmonic maps
from $\Omega\subset \Real^m$ to compact manifold $N$. Assume that
$\mathbb H(\tilde{u}_i)$ are bounded and $\tilde{u}_i$ converges weakly to $\tilde{u}$.
Set
\begin{equation*}
    \tilde{\mu}=\lim_{i\to \infty} \abs{\triangle \tilde{u}_i}^2 dx=\abs{\triangle \tilde{u}}^2 dx +\tilde{\nu},
\end{equation*}
where $\tilde{\nu}$ is the defect measure. According to Theorem 3.4 in \cite{Scheven}, $\tilde{\nu}$ is supported in $\tilde{\Sigma}$ defined as the set of points $a\in \overline{B_1}$ with
\begin{equation*}
    \liminf_{\rho\to 0} \left( \rho^{4-m}\int_{B_\rho(a)}(\abs{\triangle \tilde{u}}^2+\rho^{-2}\abs{\nabla \tilde{u}}^2)dx +\rho^{4-m}\tilde{\nu} (B_\rho(a)) \right)\geq \varepsilon_0,
\end{equation*}
where $\varepsilon_0$ is given in Corollary 2.7 of the same paper.
Moreover,      Scheven in \cite{Scheven} showed that $\tilde{\nu}$
is absolutely continuous with respect to $\mathcal
H^{m-4}\llcorner \tilde{\Sigma}$. The same proof as Theorem
\ref{thm:15} implies that
\begin{equation*}
    \lim_{r\to 0}r^{-1}\tilde{\nu}(B_r(x))
\end{equation*}
exists for $\mathcal H^{m-4}-$a.e. $x\in B_1$. Hence, by Preiss's result \cite{Big}, $\tilde{\nu}$ is $(m-4)-$rectifiable, which implies that $\tilde{\Sigma}$ is $(m-4)-$rectifiable.
\end{proof}
\section{Partially Strong Convergence and the rectifiability of the defect measure}\label{sec:two}

In this section, we pick a sequence of $\varepsilon_i$ going to
zero and write $u_i$ for $u_{\varepsilon_i}$. As proved in Lemma
\ref{lem:hong}, $u_i$ is a minimizing sequence for ${\mathbb
H}(u)$ in $W^{2,2}_{u_0}(\Omega,S^n)\cap C^0(\Omega,S^n)$. By
taking a subsequence (still denoted by $u_i)$, we have
\begin{equation*}
    \abs{\triangle u_i}^2dx\wc \mu=\abs{\triangle u}^2dx+\nu
\end{equation*}
in the sense of Radon measures. Since $u_i\wc u$ in
$W^{2,2}(\Omega,S^n)$, $\nu\geq 0$ by the Fatou  lemma. All
results and their proofs in this section depend only on the fact
that $u_i$ is a minimizing sequence of $H(u)$ in the space of
$C^0\cap W^{2,2}$.

The first result of this section is to prove that for each $x_0$
with $B_{R_0} (x_0)\subset \Omega \subset \Real^5$ for some
$R_0>0$, then there is an $\varepsilon_1>0$ such that if $\frac
1{R_0}\mu (B_{R_0}(x_0))<\varepsilon_1$, $\nu |_{B_{\frac
{R_0}2}(x_0)}\equiv 0$.

Our proof is based on an idea of Lin in \cite{Acta}. However, we
are not able to prove this for a dimension $m$ greater than $5$.
For the proof, we need a lemma.
\begin{lem} \label{lem:weneed}
    Assume that $\rho$ is a fixed positive constant and $u$ is a smooth map from $\overline{B_{\rho}}$ to $S^n$. Then there exists a positive number $\eta_1$
    such that for  any positive $\eta<\eta_1$ and $v$ defined by
    \begin{equation}\label{eqn:weneed}
        \left\{
        \begin{array}[]{ll}
            \triangle^2 v=0 & \mbox{ in } B_{\rho}\setminus B_{\rho(1-\eta)}; \\
            v=u & \mbox{ on } \partial B_{\rho}\cup \partial B_{\rho(1-\eta)}; \\
            \frac{\partial v}{\partial n}=\frac{\partial u}{\partial n} & \mbox{ on } \partial B_{\rho}\cup \partial B_{\rho(1-\eta)}
        \end{array}
        \right.
    \end{equation}
    we have
    \begin{equation*}
        \abs{v}\geq \frac{1}{2}
    \end{equation*}
    on $B_{\rho}\setminus B_{\rho(1-\eta)}$.
\end{lem}

\begin{proof}
    For simplicity, we will write $\Omega_\eta$ for $B_{\rho}\setminus B_{\rho(1-\eta)}$. For a fixed $\eta >0$,
    the solution $v$ to (\ref{eqn:weneed}) is denoted by $v_\eta$. Since $v_\eta$ is a biharmonic function, we have
    \begin{equation*}
        \int_{\Omega_\eta} \abs{\triangle v_\eta}^2 dx\leq \int_{\Omega_{\eta}} \abs{\triangle u}^2 dx\leq C\eta.
    \end{equation*}
    Here in the last inequality, we used the fact that $u$ is smooth in $\overline{B_{\rho}}$.
    Set $w_\eta=v_\eta-u$. We have
    \begin{equation*}
        \int_{\Omega_\eta} \abs{\triangle w_\eta}^2 dx\leq C\eta
    \end{equation*}
    and
    \begin{equation*}
        \left\{
        \begin{array}[]{ll}
            \triangle^2 w_\eta=\triangle^2 u & \mbox{ in } \Omega_\eta\\
            w_\eta=0 & \mbox{ on } \partial \Omega_\eta\\
            \frac{\partial w_\eta}{\partial n}=0 & \mbox{ on } \partial\Omega_\eta.
        \end{array}
        \right.
    \end{equation*}

    Since $\abs{u}=1$, it suffices for Lemma 4.1 to prove that
    \begin{equation}\label{eqn:tozero}
        \lambda_\eta=\max_{\Omega_{\eta}} \abs{w_\eta}\to 0
    \end{equation}
    as $\eta$ goes to $0$.

    Next, we prove (\ref{eqn:tozero})   by contradiction. If (\ref{eqn:tozero}) is not true, there exists a positive number $\tilde{\delta}>0$, a sequence of $\eta_i\to 0$
  and a sequence of points $p_i\in \Omega_{\eta_i}$ such that
    \begin{equation}\label{eqn:wi}
        \abs{w_{\eta_i}(p_i)}=\lambda_{\eta_i}>\tilde{\delta}.
    \end{equation}
    By a rotation if necessary, we may assume that $p_i=(0,0,0,0, p_i^5)$. Define
    \begin{equation*}
        \tilde{w}_i(\tilde x)=\frac{1}{\lambda_{\eta_i}}w_{\eta_i}(\eta_i \tilde x+p_i).
    \end{equation*}
    Let $\tilde{\Omega}_i$ be the corresponding set defined by
   \[\tilde{\Omega}_i=\{\tilde x\in \Bbb R^5 : \quad \eta_i \tilde x+p_i \in \Omega_{\eta_i} \} \]
   and we write $\tilde{\triangle}$ for the new Laplacian operator in $\tilde x$.
    \begin{equation*}
        \left\{
        \begin{array}[]{ll}
            \tilde{\triangle}^2 \tilde{w}_i =\frac{1}{\lambda_{\eta_i}}\eta_i^4 (\triangle^2 u)(\eta_i \tilde x+p_i) & \mbox{ in } \tilde{\Omega}_i \\
            \tilde{w}_i=0 & \mbox{ on }\partial\tilde{\Omega}_i \\
            \frac{\partial \tilde{w}_i}{\partial n}=0 & \mbox{ on }\partial\tilde{\Omega}_i.
        \end{array}
        \right.
    \end{equation*}
    Moreover,
    \begin{equation*}
      \int_{\tilde{\Omega}_i}\abs{\tilde{\triangle} \tilde{w}_i}^2 d\tilde x \leq C.
    \end{equation*}

    Consider two hypersurfaces $H_1$ and $H_2$ given by
    \begin{equation*}
       H_1:= \{ \tilde{x}\in \Real^5 |\quad \tilde x_5=0\}
    \end{equation*}
    and
    \begin{equation*}
       H_2:= \{ \tilde{x}\in \Real^5 | \quad  \tilde x_5=\rho\}.
    \end{equation*}
    For each large positive $K$, set
    \begin{equation*}
        D_K=\left\{\tilde {x}\in \Real^5|\quad 0\leq \tilde x_1\leq \rho, \sum_{i=2}^5 \tilde x_i^2\leq K^2\right\}.
    \end{equation*}
    We also denote the unbounded domain between $H_1$ and $H_2$ by $D_\infty$.
    There is a sequence of diffeomorphisms
    \begin{equation*}
        \Phi_i:\Real^5\to \Real^5
    \end{equation*}
    such that when $i$ is sufficiently large compared to $K$,

    (1) it maps $D_K$ to a part of the annulus $\tilde{\Omega}_i$ containing the origin in the middle;

    (2)
    \begin{equation*}
        \norm{\Phi_i-\mbox{id}}_{C^4(D_K)}\to 0
    \end{equation*}
    as $i\to \infty$.

    Fix $K$, for $i$ large, set
    \begin{equation*}
        \bar{w}_i=\tilde{w}_i\circ \Phi_i: D_K\to \Real^5.
    \end{equation*}
    Then $\bar{w}_i$ satisfies,
    \begin{equation*}
        \left\{
        \begin{array}[]{ll}
            (\tilde{\triangle}^2 \bar{w}_i\circ \Phi^{-1}_i)\circ \Phi_i = \frac{1}{\lambda_{\eta_i}}\eta_i^4 ({\triangle}^2 u)\circ \Phi_i & \mbox{ in } D_K \\
            \bar{w}_i =0 & \mbox{ on } D_K\cap (H_1\cup H_2) \\
            \pfrac{}{x_5} \bar{w}_i = ((\Phi_i)_* \pfrac{}{x_5}) \tilde{w}_i & \mbox{ on } D_K\cap (H_1\cup H_2).
        \end{array}
        \right.
    \end{equation*}
    Letting $i\to \infty$  and then letting $K\to \infty$, we obtain  a biharmonic function $w$ as a limit of $\bar{w}_i$ such that
    \begin{equation*}
        \left\{
        \begin{array}[]{ll}
            \triangle^2 w=0 & \mbox{ in } D_\infty \\
            w =0 & \mbox{ on } H_1\cup H_2 \\
            \pfrac{}{x_5} w = 0 & \mbox{ on } H_1\cup H_2.
        \end{array}
        \right.
    \end{equation*}
    We claim that from the construction,

    (1)
    \begin{equation*}
        \int_{D_\infty} \abs{\triangle w}^2 dx<C;
    \end{equation*}

    (2) $w$ is bounded but non-zero because $w$ is a limit of $\tilde{w}_i\circ \Phi_i$ and by (\ref{eqn:wi})
    $$\max_{\Omega_i} \tilde{w}_i=\tilde{w}_i(0)=1.$$

    We will see that this is a contradiction. Let $\hat{w}_i$ be a sequence of smooth functions with compact support and the same boundary condition which converges to $w$ in $W^{2,2}$ norm. Hence,
    \begin{eqnarray*}
        0  =  \lim_{i\to \infty} \int_{D_\infty} \triangle^2 w \hat{w}_i dx = \lim_{i\to \infty} \int_{D_\infty}\triangle w\triangle \hat{w}_i dx = \int_{D_\infty} \abs{\triangle w}^2 dx.
    \end{eqnarray*}
    This implies that $w$ is harmonic. It is obvious that a bounded harmonic function with zero Dirichlet boundary condition must be zero, which is a contradiction to (2).
\end{proof}

The following lemma is an elliptic estimate involving the Sobolev
space of fractional order. However,  it is not easily to find a
proper reference, so we will outline a proof. We  denote by
$\norm{\cdot}_{(s)}$  the $W^{s,2}$ Sobolev norm obtained by
complex interpolation if $s$ is not a positive integer.
\begin{lem}
    \label{lem:elliptic}
    Let $u$ be a biharmonic function on $\Omega$. Assume $u$ satisfies the boundary conditions
    \begin{equation*}
        u|_{\partial \Omega}=f,\quad \frac{\partial u}{\partial n}|_{\partial \Omega}=g.
    \end{equation*}
    Then for any $s>0$, there exists a constant $C$ depending on the dimension and $\Omega$ such that
    \begin{equation*}
        \norm{u}_{(s)}\leq C(\norm{f}_{(s-\frac{1}{2})}+\norm{g}_{(s-3/2)}).
    \end{equation*}
\end{lem}

\begin{proof}
    For $x\in \partial \Omega$, take two open neighborhoods of $x$, $U,V$ such that
    \begin{equation*}
        x\in V \subset \overline{V}\subset U.
    \end{equation*}
    Assume that $\overline{U}\cap \Omega$ is diffeomorphic to $B_1^+=\{x\in B_1|\, x_1\geq 0\}$. Let $\varphi$ be a smooth  cut-off function
    satisfying (1) $\varphi(y)\equiv 0$ for $y\notin U$; (2) $\varphi(y)\equiv 1$ for $y\in V$. A direct computation implies
    \begin{equation}
        \left\{
        \begin{array}[]{l}
            \triangle^2 (\varphi u)= D^3u \# D\varphi +D^2 u\# D\varphi +Du\#D^3\varphi+D^4\varphi \\
            (\varphi u)|_{\partial \Omega}=\varphi f \\
            \pfrac{}{n}(\varphi u)|_{\partial \Omega} = \varphi g+\pfrac{\varphi}{n}f.
        \end{array}
        \right.
    \end{equation}
    Here $D^k$ means partial derivatives of order $k$ and $D^3u \# D\varphi$ means linear combinations of the product of $D^3u$ and $D\varphi$ and so on.
    By \cite{Ark}, we have
    \begin{equation}\label{eqn:V}
        \norm{u}_{(s),V}\leq C(\varphi) (\norm{u}_{(s-1),U}+\norm{f}_{(s-1/2),\partial\Omega\cap U}+\norm{g}_{(s-3/2),\partial \Omega\cap U}).
    \end{equation}

    Since the boundary $\partial\Omega$ is compact, we can find a finite number of points $x_1,\cdots, x_k$ such that $\partial \Omega$ is covered by $V_i$'s. Adding (\ref{eqn:V}) up for all $V_i$ and using the interior estimate, we obtain
    \begin{equation}\label{eqn:ark}
        \norm{u}_{(s)}\leq C(\norm{f}_{(s-1/2)} +\norm{g}_{(s-3/2)}+\norm{u}_{(s-1)}).
    \end{equation}

    Next, we claim that the $\norm{u}_{(s-1)}$ term in the right hand side is not necessary for our case. This is proved by contradiction. If otherwise, there exists a sequence of $f_i,g_i,u_i$ such that

    (1) $u_i$ is a biharmonic function  on $\Omega$ with $u_i|_{\partial \Omega}=f_i$ and $\frac{ \partial u_i}{\partial n}|_{\partial \Omega}=g_i$;

    (2) assume by scaling that $\norm{f_i}_{(s-1/2)}+\norm{g_i}_{s-3/2}=1$;

    (3) $\norm{u_i}_{(s)}\geq i$.

    It follows from (3) and (\ref{eqn:ark}) that $\lim_{i\to \infty} \norm{u_i}_{(s-1)}=\infty$. Let $\lambda_i$ be $\norm{u_i}_{(s-1)}$ and
    set $\tilde{u}_i=\frac{u_i}{\lambda_i}$, $\tilde{f}_i=\frac{f_i}{\lambda_i}$ and $\tilde{g}_i=\frac{g_i}{\lambda_i}$. Using (\ref{eqn:ark}) again,
    we have that $\norm{u_i}_{(s)}$  is bounded. Therefore, $u_i$ converges weakly in $W^{s,2}(\Omega)$ to a biharmonic function with homogeneous boundary conditions.
    On one hand, due to the compactness of embedding from $W^{s,2}$ to $W^{s-1,2}$, we have $\norm{u}_{(s-1)}=1$. On the other hand,  the only biharmonic function
     with homogeneous boundary conditions is zero. This is a contradiction.
\end{proof}
\begin{prop}\label{thm:2}
    Let $u_i$ be the sequence defined in Theorem 1.3. Then there exists a positive constant $\varepsilon_1$ such that if we set
    \begin{equation*}
        \Sigma_1=\bigcap_{R>0}\{x\in \Omega| \, B_R(x)\subset \Omega, R^{-1}\mu(B_R(x))\geq \varepsilon_1\},
    \end{equation*}
    then
    \begin{equation*}
        \mbox{spt}\, \nu\subset \Sigma_1.
    \end{equation*}
\end{prop}

\begin{proof}
    $\varepsilon_1$ is determined during the proof. We can require that $\varepsilon_1<\varepsilon_0$ such that $u$ is smooth away from $\Sigma_1$.
    Hence, if $x_0\notin \Sigma_1$, then there is an $R>0$ such that $u$ is smooth in $\overline{B_R(x_0)}$ and
    \begin{equation*}
        \liminf_{i\to \infty} R^{-1}\int_{B_R(x_0)}\abs{\triangle u_i}^2 dx< \varepsilon_1.
    \end{equation*}
    It suffices to show that $\nu\equiv 0$ in $B_{R/3}(x_0)$. Assume $x_0$ is the origin and $R=1$.

    It is obvious that
    \begin{equation*}
        \int_{B_{0.9}} \abs{\nabla^2 u_i}^2+\abs{\nabla u_i}^4 dx  \leq C \varepsilon_1.
    \end{equation*}
    Pick $\rho\in (1/2,2/3)$ such that
    \begin{equation}\label{eqn:pick}
        \int_{\partial B_{\rho}}\abs{\nabla^2 u_i}^2 +\abs{\nabla u_i}^4 d\sigma\leq C\varepsilon_1
    \end{equation}
    for infinitely many $i$'s. Assume by taking subsequence that (\ref{eqn:pick}) is true for all $i$.

    Since $\abs{\triangle u_i}^2dx=\mu_i\wc \mu=\abs{\triangle u}^2 dx+\nu$ in $B_{1}$ as Radon measures, to show $\nu\equiv 0$ in $B_{1/3}$, it suffices to show
    \begin{equation*}
        {\mathbb H}(u_i,B_{\rho})\leq {\mathbb H}(u,B_{\rho})+\delta
    \end{equation*}
    for any $\delta>0$ and all sufficiently large $i$'s.

    To do so, we use the fact that $u_i$ is a minimizing sequence in $W^{2,2}_{u_0}(\Omega,S^n)\cap C^0(\Omega,S^n)$. We shall construct a new sequence $\{\tilde{u}_i\}$ in $W^{2,2}_{u_0}(\Omega,S^n)\cap C^0(\Omega,S^n)$ such that

    (a) $\tilde{u}_i=u_i$ on $\Omega\setminus B_{\rho}$;

    (b) $\tilde{u}_i=u$ on $B_{\rho(1-\eta)}$ for a very small $\eta$ to be determined in the following proof;

    (c)
    \begin{equation*}
        \int_{B_{\rho}\setminus B_{\rho(1-\eta)}}\abs{\triangle \tilde{u}_i}^2 dx\leq C(\eta)<\delta.
    \end{equation*}

    Given this new sequence of $\tilde{u}_i$, by the definition of minimizing sequence,
    \begin{equation*}
        \liminf_{i\to \infty} {\mathbb H}(\tilde{u}_i)\geq \lim_{i\to \infty} {\mathbb H}(u_i).
    \end{equation*}
    Due to (a), (there is no guarantee that $\lim_{i\to \infty} \int_{B_{\rho}}\abs{\triangle u_i}^2 dx$ exists, but we can always take a subsequence such that this is true. This does not affect the result that $\nu=0$.)
    \begin{equation*}
        \liminf_{i\to \infty} {\mathbb H}(\tilde{u_i},B_{\rho})\geq \lim_{i\to \infty} {\mathbb H}(u_i,B_{\rho}).
    \end{equation*}
    Therefore, for all sufficiently large $i$,
    \begin{equation}\label{eqn:smallenergy}
        {\mathbb H}(u_i,B_{\rho})\leq {\mathbb H}(u,B_{\rho})+C(\eta).
    \end{equation}

    The above discussion shows that Theorem \ref{thm:2} follows from a construction of $\tilde{u}_i$ satisfying (a), (b) and (c).
    Due to (a) and (b), it suffices to define $\tilde{u}_i$ in $B_{\rho}\setminus B_{\rho(1-\eta)}$. The construction consists of several steps.

    {\bf Step one.} Let $v$ be the solution of the boundary value problem
    \begin{equation}\label{eqn:eta}
        \left\{
        \begin{array}[]{ll}
            \triangle^2 v=0 & \mbox{ in } B_{\rho}\setminus B_{\rho(1-\eta)}; \\
            v=u & \mbox{ on } \partial B_{\rho}\cup \partial B_{\rho(1-\eta)}; \\
            \frac{\partial v}{\partial n}=\frac{\partial u}{\partial n} & \mbox{ on } \partial B_{\rho}\cup \partial B_{\rho(1-\eta)}.
        \end{array}
        \right.
    \end{equation}
    Here we require $\eta$ to be smaller than the $\eta_1$ given by Lemma \ref{lem:weneed}. There will be another restriction to $\eta$ in Step five. The point is that $\eta$ doesn't depend on $i$.

{\bf Step two.} Define $v_i$ as the biharmonic extension of $u$ and $\tilde{u}_i$ as follows.
    \begin{equation*}
        \left\{
        \begin{array}[]{ll}
            \triangle^2 v_i=0 & \mbox{ in } B_{\rho}\setminus B_{\rho(1-\eta)}; \\
            v=u & \mbox{ on } \partial   B_{\rho(1-\eta)}; \\
            v=u_i & \mbox{ on } \partial B_{\rho}; \\
            \frac{\partial v}{\partial n}=\frac{\partial u}{\partial n} & \mbox{ on }  \partial B_{\rho(1-\eta)};\\
            \frac{\partial v}{\partial n}=\frac{\partial u_i}{\partial n} & \mbox{ on } \partial B_{\rho}.
        \end{array}
        \right.
    \end{equation*}
    We need to prove some estimates of $v_i$. Recall that both $u_i$ and $u$ are bounded in $W^{2,2}$(in fact $u$ is smooth.)
    The restriction of $W^{2,2}$ function to a hypersurface belongs to $W^{1.5,2}$ and
    \begin{equation*}
      \pfrac{u_i}{n}\in W^{0.5,2}(\partial B_{\rho}).
    \end{equation*}
    By Lemma \ref{lem:elliptic}, we have
    \begin{eqnarray}\label{eqn:w22}
      &&\norm{v_i}_{{W}^{2,2}(B_{\rho}\setminus B_{\rho(1-\eta)})}\\ \nonumber
      &\leq& C( \norm{u_i}_{W^{1.5,2}(\partial B_{\rho})}+\norm{\pfrac{u_i}{n}}_{W^{0.5,2}(\partial B_{\rho})}) +C(u)\\ \nonumber
      &\leq& C \norm{u_i}_{W^{2,2}(B_{\rho}\setminus B_{\rho(1-\eta)})}+C(u) \\ \nonumber
      &\leq& C(u).
    \end{eqnarray}

    Moreover, we can obtain better estimate if we take into account the special choice of $\rho$. Due to (\ref{eqn:pick}), we have
    \begin{equation*}
      \norm{u_i}_{W^{2,2}(\partial B_{\rho})}\leq C(\varepsilon_1)
    \end{equation*}
    and
    \begin{equation*}
      \norm{\partial_n u_i}_{W^{1,2}(\partial B_{\rho})}\leq C(\varepsilon_1).
    \end{equation*}
    This combined with the fact that $u$ is smooth implies that (by Lemma \ref{lem:elliptic} again)
    \begin{equation}\label{eqn:step2}
        \norm{v_i}_{{W}^{2.5,2}(B_{\rho}\setminus B_{\rho(1-\eta)})}\leq C(n,\varepsilon_1,u).
    \end{equation}

    {\bf Step three.} We need to use the Poisson formula to show that there exists a thin layer given by $B_{\rho}\setminus B_{\rho(1-\lambda)}$ for some $\lambda<< \eta$ such that the image of $v_i$ stay near the sphere in this layer.

For simplicity, we may assume without loss of generality that $\rho=1$ and $\eta=1/2$ (in this step only). Since $\rho$ and $\eta$ is fixed, this doesn't affect the proof. We will show there exists a small $\lambda>0$ such that
\begin{equation*}
  \abs{v_i}>1/2
\end{equation*}
on
\begin{equation*}
  B_{1}\setminus B_{1-\lambda}.
\end{equation*}
According to Green's formula for the biharmonic equation (\cite{green}),
\begin{equation}
  v_i(x)= \int_{\partial B_{1/2}\cup \partial B_{1}} K_0(x,y) v_i(y) d\sigma_y
 +\int_{\partial B_{1/2}\cup \partial B_{1}} K_1(x,y) \partial_n v_i(y) d\sigma_y.
  \label{eqn:poisson}
\end{equation}
Following \cite{hly}, set $\xi_0=\frac{x}{\abs{x}}$ and $r=1-\abs{x}$. Since we will only consider estimate near $\partial B_1$, we may require $x\in B_1\setminus B_{4/5}$. Therefore, $r$ is the distance from $x$ to $\partial (B_1\setminus B_{1/2})$.
Here is an estimate on $K_0$ and $K_1$ from \cite{green},
\begin{equation}
  \abs{K_0(x,y)}\leq C \frac{r^2}{d^6(x,y)}
  \label{eqn:K0}
\end{equation}
and
\begin{equation}\label{eqn:K1}
    \abs{K_1(x,y)}\leq C\frac{r^2}{d^5(x,y)},
\end{equation}
for $y\in \partial B_1\cup \partial B_{1/2}$ and $x\in B_1\setminus B_{4/5}$.
For some $k>1$ with $kr\leq \frac{1}{4}$, we write
\begin{equation*}
    v_{kr,\xi_0}=\frac{1}{\abs{\partial B_{1}\cap B_{kr}(\xi_0)}}\int_{\partial B_{1}\cup B_{kr}(\xi_0)} v_i d\sigma.
\end{equation*}
Using the Poincar\'e inequality, we see
\begin{eqnarray}\label{eqn:poincare}
    &&\frac{1}{\abs{\partial B_{1}\cap B_{kr}(\xi_0)}}\int_{\partial B_{1}\cap B_{kr}(\xi_0)} \abs{v_i-v_{kr,\xi_0}}d\sigma \\ \nonumber
    &\leq& C\norm{\nabla u_i}_{L^{4}(\partial B_{1}\cap B_{kr}(\xi_0))}\leq C  \varepsilon_1^{1/4}.
\end{eqnarray}
Hence $\abs{v_{kr,\xi_0}-1}\leq C\varepsilon_1^{1/4}$.
Since the constant function $v_{kr,\xi_0}$ is a biharmonic function with constant Dirichlet boundary value and zero Neumann boundary value, we have
\begin{eqnarray}\label{eqn:poisson2}
  v_i(x)-v_{kr,\xi_0}&=&  \int_{\partial B_{1/2}\cup \partial B_{1}} K_0(x,y) (v_i(y)-v_{kr,\xi_0}) d\sigma_y + \\
  \nonumber &&
 \int_{\partial B_{1/2}\cup \partial B_{1}} K_1(x,y) \partial_n v_i(y) d\sigma_y.
\end{eqnarray}
To estimate the first integral, we divide the integral domain into two parts,
\begin{equation*}
    \Omega_1=\partial B_{1}\cap B_{kr}(\xi_0) \mbox{ and } \Omega_2=(\partial B_{1}\setminus B_{kr}(\xi_0))\cup \partial B_{1/2}.
\end{equation*}
$\Omega_2$ is further divided into two parts. Then we estimate
\begin{eqnarray*}
     & &
     \int_{\Omega_2}\abs{K_0(x,y)(v_i-v_{kr,\xi_0})}d\sigma_y\\
    &\leq& \left( \int_{\Omega_2\cap B_{1/2}(\xi_0)}+\int_{\Omega_2\setminus B_{1/2}(\xi_0)}\right) \abs{K_0(x,y)(v_i-v_{kr,\xi_0})}
    d\sigma_y.
\end{eqnarray*}
For $y\in \Omega_2\setminus B_{1/2}(\xi_0)$, we note
\begin{equation*}
    \abs{K_0(x,y)}\leq Cr^2.
\end{equation*}
Hence,
\begin{equation}\label{eqn:f1}
    \int_{\Omega_2\setminus B_{1/2}(\xi_0)} \abs{K_0(x,y)(v_i-v_{kr,\xi_0})}d\sigma_y\leq Cr^2.
\end{equation}
Using the fact that $v_i$ and $v_{kr,\xi_0}$ are bounded, we have
\begin{eqnarray*}\label{eqn:f2}
    \int_{\Omega_2\cap B_{1/2}(\xi_0)} \abs{K_0(x,y)(v_i-v_{kr,\xi_0})} d\sigma_y &\leq& \int_{kr}^{1/2} \frac{Cr^2}{t^6} t^3dt \\
    &\leq& \frac{C}{k^2}-Cr^2.  \nonumber
\end{eqnarray*}
Here $t$ is the distance between $\xi_0$ and $y$ on the sphere $\partial B_1$ and we estimate $d(x,y)$ from below by $Ct$. We add (\ref{eqn:f1}) and (\ref{eqn:f2}) to get
\begin{equation*}
    \int_{\Omega_2}\abs{K_0(x,y)(v_i-v_{kr,\xi_0})}d\sigma_y\leq Cr^2+\frac{C}{k^2}.
\end{equation*}
For all $y\in \partial B_1\cup \partial B_{1/2}$ and $x\in
B_1\setminus B_{4/5}$, we see
\begin{equation*}
    d(x,y)\geq r.
\end{equation*}
If $y$ is a point in $\Omega_1$,  by (\ref{eqn:poincare}), we have
\begin{eqnarray*}
    \abs{\int_{\partial B_{1}\cap B_{kr}(\xi_0)}K_0(x,y)(u_i-v_{kr,\xi_0})}&\leq& \frac{C}{r^4}\int_{\partial B_{1}\cap B_{kr}(\xi_0)}\abs{u_i-v_{kr,\xi_0}}d\sigma_y \\
    &\leq& Ck^4 \varepsilon_1^{1/4}.
\end{eqnarray*}
The second integral in (\ref{eqn:poisson2}) is estimated similarly.
\begin{eqnarray*}
    \abs{\int_{\Omega_2}K_1(x,y)\partial_n v_i \sigma_y}&\leq& \left(\int_{\Omega_2\cap B_{1/2}(\xi_0)}+\int_{\Omega_2\setminus B_{1/2}(\xi_0)}\right) \abs{K_1(x,y)\partial_n v_i}d\sigma_y \\
    &\leq& \left( \int_{\Omega_2\cap B_{1/2}(\xi_0)}\abs{K_1}^{4/3} \right)^{3/4} \left( \int_{\Omega_{2}\cap B_{1/2}(\xi_0)} \abs{\partial_n v_i}^4 \right)^{1/4} \\
    &&+ Cr^2\int_{\Omega_2\setminus B_{1/2}(\xi_0)} \abs{\partial_n v_i}d\sigma_y \\
    &\leq& C \left( r^{8/3} \int_{kr}^{1/2} \frac{1}{t^{20/3}}t^3dt \right)^{3/4}+Cr^2 \\
    &\leq& C\left( \frac{1}{k^{8/3}}-Cr^{8/3} \right)^{3/4} +Cr^2 \\
    &\leq& \frac{C}{k^2}+Cr^2,
\end{eqnarray*}
where we used (\ref{eqn:K1}), (\ref{eqn:pick}) and the H\"older inequality.
On the other hand,
\begin{eqnarray*}
    \abs{\int_{\Omega_1}K_1(x,y){\partial_n u_i}}d\sigma &\leq& \frac{C}{r^3}\int_{\Omega_1}\abs{\partial_n u_i}d\sigma_y \\
    &\leq& Ck^3\left( \int_{\Omega_1} \abs{\partial_n u_i}^4 d\sigma_y \right)^\frac{1}{4}\leq Ck^3 \varepsilon_1^{1/4}.
\end{eqnarray*}
In summary, we have
\begin{equation*}
    \abs{v_i(x)-v_{kr,\xi_0}}\leq \frac{C}{k^2}+C\varepsilon_1^{1/4}(k^4+k^3).
\end{equation*}
We can choose $k$ large so that $\frac{C}{k^2}<\frac{1}{10}$ and then choose $\varepsilon_1$ small so that $C\varepsilon_1^{1/4}(k^4+k^3) + 1-\abs{v_{kr,\xi_0}}<\frac{1}{10}$. Hence, if we set $\lambda=\frac{1}{4k}$, we have
\begin{equation*}
    \abs{v_i}\geq \frac{3}{4},
\end{equation*}
for any point $x\in B_{1}\setminus B_{1-\lambda}$ with $r<\lambda$.

{\bf Step four.} We will establish an estimate of $v_i$ on
    \begin{equation*}
        B_{\rho(1-\lambda/2)}\setminus B_{\rho(1-\eta)}.
    \end{equation*}
    Due to the interior estimate for biharmonic functions and (\ref{eqn:w22}) in Step two,
    \begin{equation*}
        \norm{v_i}_{C^l(B_{\rho(1-\lambda/2)}\setminus B_{\rho(1-\eta+\lambda/2)})}\leq C(l).
    \end{equation*}
    Given this, the elliptic boundary value problem on $B_{\rho(1-\eta/2)}\setminus B_{\rho(1-\eta)}$ implies
    \begin{equation*}
        \norm{v_i}_{C^l(B_{\rho(1-\eta/2)}\setminus B_{\rho(1-\eta)})}\leq C(l,u),
    \end{equation*}
    since both boundary values are now very smooth.

    Combining the result of Step three with the result of Step four, we see that for $i$ sufficiently large, the image of $v_i$ stay in the neighborhood of $S^n$, so
     we define
    \begin{equation*}
        \tilde{u}_i=\frac{v_i}{\abs{v_i}}
    \end{equation*}
    on $B_{\rho}\setminus B_{\rho(1-\eta)}$.

    {\bf Step five.} It remains to check (\ref{eqn:smallenergy}).
    \begin{eqnarray*}
        {\mathbb H}(\tilde{u}_i,B_{\rho}\setminus B_{\rho(1-\eta)}) &\leq& C{\mathbb H}(v_i,B_{\rho}\setminus B_{\rho(1-\eta)}) \\
        &\leq & 2C {\mathbb H}(v,B_{\rho}\setminus B_{\rho(1-\eta)}) \\
        &\leq &  2C {\mathbb H}(u,B_{\rho}\setminus B_{\rho(1-\eta)}) \\
      &\leq& C(u) \eta.
    \end{eqnarray*}
    Here for the second inequality above, we apply (\ref{eqn:step2}) and the argument in Step four to show the energy of $v_i$ converges to that of $v$.
   Thus, we  make this smaller than $\delta$ if we   choose $\eta$ smaller than $\eta_2(u,\delta)>0$.
\end{proof}

We are now ready to prove Theorem \ref{thm:third}:

\begin{proof}[Proof of Theorem   \ref{thm:third}]
    The proof about finite $1-$dimension Hausdorff measure of the singular set $\Sigma_1$ is standard (see the proof of Theorem 3.4 in \cite{Scheven}). To show $\Sigma_1$ is relatively closed, let $x_i$ be a sequence in $\Sigma_1$ such that $x_i\to x\in \Omega$. Let $R>0$ be such that $B_R(x)\subset \Omega$, it suffices to show
    \begin{equation}\label{eqn:closed}
        R^{4-m}\mu (B_R(x))\geq \varepsilon_1.
    \end{equation}
    Pick any $r<R$. For $i$ sufficiently large, $B_r(x_i)\subset B_R(x)\subset \Omega$. Hence
    \begin{equation*}
        r^{4-m}\mu(B_r(x_i))\geq \varepsilon_1.
    \end{equation*}
    This implies that
    \begin{equation*}
        R^{4-m}\mu (B_R(x))\geq \left( \frac{R}{r} \right)^{4-m} r^{4-m} \mu(B_{r}(x_i))\geq \left( \frac{R}{r} \right)^{4-m}\varepsilon_1.
    \end{equation*}
    Since $r$ can be arbitrarily close to $R$, (\ref{eqn:closed}) is true, hence (1) is proved.

    By Theorem \ref{thm:first} and Proposition \ref{thm:2}, we
    obtain
    \begin{equation*}
        \mbox{spt}\, \nu\cup \mbox{sing}\, u\subset \Sigma_1.
    \end{equation*}
    If $x\notin \mbox{spt}\,\nu \cup \mbox{sing}\, u$, there is $R>0$ such that
    \begin{equation*}
        B_R(x)\cap \mbox{sing}\, u =\emptyset
    \end{equation*}
    and
    \begin{equation*}
        \nu(B_R(x))=0.
    \end{equation*}
    Hence, for $r<R$,
    \begin{equation*}
        r^{-1}\mu(B_r(x))=r^{-1}\int_{B_r(x)}\abs{\triangle u}^2 dx.
    \end{equation*}
    By the smoothness of $u$ in $B_{R/3}(x)$,
    \begin{equation*}
        r^{-1}\mu(B_r(x))\to 0
    \end{equation*}
    when $r$ goes to zero. This implies $x\notin \Sigma_1$, which proves (2).

    For (3), Lemma \ref{lem:41} implies that $\mu$, hence $\nu$ is absolutely continuous with respect to $\mathcal H^{1}\llcorner \Sigma_1$. By the Radon-Nikodym theorem, one has
    \begin{equation*}
        \mu|_{\Sigma_1}=\Theta(x)\mathcal H^{1}\llcorner \Sigma_1,
    \end{equation*}
    for $\mathcal H^1-$a.e. $x\in \Sigma_1$. By Corollary 3.2.3 in \cite{Ziemer},
    \begin{equation*}
        \nu|_{\Sigma_1}=\Theta(x)\mathcal H^{1}\llcorner \Sigma_1,
    \end{equation*}
    for $\mathcal H^1-$a.e. $x\in \Sigma_1$. The estimates of $\Theta(x)$ follows
    from the fact that
    \begin{equation*}
        \varepsilon_0\leq r^{-1}\mu(B_r(x))\leq C,
    \end{equation*}
    for $\mathcal H^{1}-$a.e. $x\in \Sigma_1$.
Thus (1) and (2) is proved.

   For (3),
according to Preiss, it suffices to show that for $\nu$ almost every $x$,
\begin{equation*}
    0<\lim_{r\to 0} r^{-1}\nu(B_r(x))<\infty.
\end{equation*}
This is nothing but Theorem \ref{thm:15}. By Theorem \ref{thm:15},
we have
\begin{equation*}
    \lim_{r\to 0}r^{-1}\mu(B_r(x))
\end{equation*}
exists except for a set of $\mathcal H^1$ measure zero. Since $\nu$ is absolutely continuous with respect to $\mathcal H^1\llcorner \Sigma_1$, this is true for $\nu-$a.e. $x\in \Sigma_1$.
\end{proof}

\begin{acknowledgement}
 {The research  of the authors was supported by the Australian Research Council
grant DP0985624.}
\end{acknowledgement}

\end{document}